\documentclass[11pt]{article}
    \title{On the Curvature of Regge Metrics}
    \author{Evan S. Gawlik\thanks{Department of Mathematics and Computer Science, Santa Clara University, ORCID 0000-0002-3244-3426}, Jack McKee\thanks{Department of Mathematics, University of Hawai`i at M$\overline{\mbox{a}}$noa, ORCID 0009-0004-1865-9089, jmckee@math.hawaii.edu}    }
    \usepackage[bindingoffset=0.2in,
            left=1in,
            right=1.5in,
            top=1in,
            bottom=1in,
            footskip=.25in]{geometry}
    \usepackage{amsmath}
    \usepackage{amsfonts}
    \usepackage{amsthm}
    \usepackage{amssymb}
    \usepackage{xcolor}
    \usepackage{tikz}
    \usetikzlibrary{cd}
    \usepackage[normalem]{ulem}
    \newtheorem{theorem}{Theorem}
    \newtheorem{definition}[theorem]{Definition}
    \newtheorem{corollary}[theorem]{Corollary}
    
        \newtheorem{proposition}[theorem]{Proposition}

    \newtheorem{lemma}[theorem]{Lemma}
    \newtheorem{remark}[theorem]{Remark}
    \newcommand{\pdiff}[2]{\frac{\partial #1}{\partial #2}}
    \newcommand{\tdiff}[2]{\frac{d #1}{d #2}}
    \newcommand{\iprod}[1]{\left\langle #1 \right\rangle}
    \newcommand{\Two}{\mathrm{I\!I}}

\usepackage{graphicx}
\usepackage[normalem]{ulem}

\usepackage[style=numeric-comp,giveninits=true,url=false,doi=false,isbn=false,date=year,maxbibnames=100,backend=bibtex]{biblatex}
\bibliography{connection_form.bib}

\usepackage{hyperref}

\begin{document}

\maketitle 

\begin{abstract}
We use moving frame techniques to derive a notion of curvature for a class of piecewise-smooth Riemannian metrics called Regge metrics, showing that it is a measure that simultaneously satisfies the (weak) Cartan structure equations and the appropriate gauge transformation law. It turns out that this distributional curvature is equivalent to existing notions of densitized distributional curvature. We also investigate more closely the two-dimensional case, where we prove the Gauss-Bonnet theorem for Regge metrics.
\end{abstract}

\section{Introduction}

Since Tullio Regge's introduction of what are today called Regge metrics---discontinuous metrics with tangential-tangential continuity---they have found widespread use in numerical models of general relativity~\cite{Regge,Li}, continuum mechanics~\cite{neunteufel2021avoiding,Li,neunteufel2024hellan}, and more. In many of these applications, one of the key features of a Regge metric is that curvature measures can be defined which converge in measure to their smooth counterparts. Regge's original paper \cite{Regge} discussed piecewise-constant metrics on simplicial meshes, where the scalar curvature is given simply by the angle defect around codimension-2 interfaces. Further investigations by Cheeger, M{\"u}ller, and Schrader \cite{cheeger1984curvature} proved that a broader class of curvatures, called the Lipschitz-Killing curvatures, converges in measure to their smooth counterparts in the piecewise-flat setting.

Later developments involved proving convergence results for Gauss, scalar, and Einstein curvature measures on higher-order Regge metrics, where the metric is piecewise-smooth rather than only piecewise-flat but retains the jump conditions on simplex boundaries \cite{gawlik2020high,BKGa22,GaNe22,gawlik2023finite}. The angle defect remains part of all of these curvature measures, but extra terms involving the curvature on element interiors and the jump in mean curvature/second fundamental form over element boundaries also appear. Only recently have convergence results been proved for a distributional version of the full Riemann curvature tensor in arbitrary dimension \cite{gopalakrishnan2023analysis}.

In this paper we will focus not on the numerical convergence of distributional curvature, but rather on deriving the ``correct" definition of distributional curvature from first principles. The formulas (\ref{distcurvfinal}) and (\ref{distcurvendofinal}) are equivalent to that for the densitized distributional Riemann curvature described in \cite{gopalakrishnan2023analysis}.  We are not the first to pursue a derivation of the correct definition from first principles, especially for the scalar curvature in the piecewise-flat setting; see \cite{christiansen2024definition,christiansen2011linearization,bernig2002scalar,bernig2006curvature} and the references therein for other perspectives on this.

Our starting point is the structure equations on the orthonormal frame bundle, about which some background is provided in the following section. We take them to be the ``ground truth'' that define the distributional curvature. The question is, then, which moving frames $f$ can be used to obtain a distributional curvature functional $f^*\Omega_{\mathrm{dist}}$ that is actually a measure (or more precisely an order-0 current), and that has the correct gauge transformation law?  We argue that the vectors comprising such a frame $f$ need to possess single-valued normal and tangential components on codimension-1 interfaces, forcing them to be discontinuous at (generic) codimension-2 interfaces.  A frame $f$ with this property will be called ``compatible'' if it satisfies a few regularity hypotheses and topological constraints detailed in Definition~\ref{compatibilitydef}. We argue that compatible frames are the correct generalization of smooth orthonormal frames. Section \ref{compatframesection} is a discussion motivating the definition of compatible frames. Sections \ref{integrationbypartssection}-\ref{angledefectsection} then use this definition to derive an expression for the distributional curvature functional $f^*\Omega_{\mathrm{dist}}$.

Importantly, a core part of the definition of a compatible frame involves \emph{blow-ups} of polytopes, essentially to ensure that the frame is regular enough to permit integration by parts, even though it is discontinuous. The idea of using blow-ups to define geometric invariants of Regge metrics is not new, appearing in \cite[p.~2]{BERCHENKOKOGAN2025100529} for much the same reason why we use it here. In Section \ref{distcurvsection}, after deriving properties that such a frame must have, we use them to obtain an expression for $f^*\Omega_{\mathrm{dist}}$ that behaves as it should. We use blow-ups to help compute $f^*\Omega_{\mathrm{dist}}$ in Sections \ref{integrationbypartssection}-\ref{angledefectsection}, and we also use them in Section~\ref{deframesection} to derive a frame-independent expression for the distributional curvature, which is more practical for real computations, especially on manifolds that are not necessarily parallelizable. 

Our main results are stated in Theorems~\ref{distcurvthrm} and~\ref{Rdistthrm}.  Roughly speaking, they say that if an $n$-dimensional manifold $M$ is equipped with a (curvilinear) polyhedral mesh, a Regge metric $g$, and a compatible frame $f$, then the structure equations for the connection one-form and curvature two-form can be given meaning in a distributional sense.  Furthermore, the distributional curvature two-form $f^*\Omega_{\mathrm{dist}}$, when reinterpreted as an $\mathrm{End}(TM)$-valued two-form, is a functional $\hat{R}_{\mathrm{dist}}$ that acts on any $\mathrm{End}(TM)$-valued $(n-2)$-form $\hat{\phi}$ with sufficient regularity via
\[
\iprod{\iprod{\hat{R}_{\mathrm{dist}},\hat{\phi}}} = \sum_{T \subseteq M} \int_{\mathring{T}} \iprod{\hat{R} \wedge \hat{\phi}} - \sum_{\mathring{e} \subset \mathring{M}} \int_{\mathring{e}} \left[\!\left[\iprod{\hat{\Two}_e \wedge i_{\mathring{e}}^*\hat{\phi}} \right]\!\right] + \sum_{\mathring{p} \subset \mathring{M}}\int_{\mathring{p}}\iprod{\hat{\Theta}_p \wedge i_{\mathring{p}}^*\hat{\phi}}.
\]
Here, the sums over $T$, $e$, and $p$ are sums over polytopes of codimension 0, 1, and 2, respectively, and the maps $i_{\mathring{e}}^*$ and $i_{\mathring{p}}^*$ are pullbacks under the inclusions $\mathring{e} \hookrightarrow M$ and $\mathring{p} \hookrightarrow M$.
The notation $[\![\cdot]\!]$ specifies the jump of a multi-valued quantity over the submanifold $e = T \cap T'$, which in this case is simply the difference $\iprod{\hat\Two_e^T \wedge i_{\mathring{e}}^*\hat\phi|_T} - \iprod{\hat\Two_e^{T'} \wedge i_{\mathring{e}}^*\hat\phi|_{T'}}$.  The quantities $\hat{R}$, $\hat{\Two}_e^T$, and $\hat\Theta_p$ are $\mathrm{End}(TM)$-valued $2$-forms, $1$-forms, and $0$-forms that encode the curvature tensor, second fundamental form, and angle defect, respectively. The operation $\iprod{\cdot \wedge \cdot}$ takes a pair of $\mathrm{End}(TM)$-valued forms and wedges their form parts and applies a nondegenerate pairing to their endomorphism parts.

Constructing compatible frames is nontrivial, and we provide an existence proof in Section \ref{constructingcompatibleframessection}. Lastly, in Section \ref{gengaussbonnetsection} we investigate the specialization to two dimensions, where the Gauss curvature measure can be defined independently of the frame, and we prove that there is a suitable generalization of the Gauss-Bonnet theorem.

Our results are stated for extremely general oriented parallelizable manifolds and meshes. This is partly because the presentation is actually not much more complicated for a general polyhedral mesh, since many delicate analytical conditions for Stokes' theorem must be dealt with even in the simplest cases, and existing literature treating this subject has already been developed to a very high level of generality.  One caveat is that we rely on the existence of blow-ups of polytopes which are also polytopes, but we have only been able to locate an existence theorem for blow-ups of simplices in the literature. An explanation of the types of manifolds we use in this paper, and references to relevant literature, can be found in the appendix. 

Some opportunities for extensions of our results are immediately apparent. The method of moving frames generalizes quite simply to indefinite (also called pseudo-Riemannian) metrics and to more general geometries as well. Lemma \ref{frameextthrm} in Section \ref{constructingcompatibleframessection} is already stated for arbitrary pseudo-Riemannian metrics. While the conditions we set for compatible frames make key use of the particular properties of the structure group $O(n)$, it is clear where the dependence lies and what would constitute removing it. Essentially the difficulty will lie in the jump conditions at codimension-2 polytopes, and correspondingly in the angle defect terms of the distributional curvature equation. These must be phrased in terms of another one-parameter group adapted to the geometry. We use the integration theory of differential forms whenever possible, avoiding metric dependence. 

\subsection{Background: Geometry in a Moving Frame}\label{backgroundsection}
In the method of moving frames~\cite{Cartan-For-Beginners}, one considers the geometry of a manifold by finding general constructions in the frame bundle, and then choosing an adapted frame which is most convenient for calculations.

Consider an oriented, parallelizable polyhedral $n$-manifold $M$ furnished with a smooth Riemannian metric $g$. The \emph{frame bundle} of $M$, denoted $\mathcal{F}_{GL}(M)$, is the sub-bundle of $TM \times \dots \times TM = (TM)^n$ such that the fiber over each point $x \in M$ is the set of ordered bases $F = (F_1,\dots,F_n)$ for $T_xM$. The \emph{orthonormal frame bundle} of $M$, denoted $\mathcal{F}_O(M)$, is the sub-bundle of $\mathcal{F}_{GL}(M)$ such that the fiber over each point is the set of ordered bases such that $\iprod{F_i,F_j} = \delta_{ij}$ for all $i,j$. A \emph{frame} is a section of the frame bundle, and an \emph{orthonormal frame} is a section of the orthonormal frame bundle.

There is a right action of $GL(n)$ on each fiber of $\mathcal{F}_{GL}(M)$ defined by $(F \cdot h)_i =  F_jh^j_i$. (To be consistent with Einstein summation notation, the entry of a matrix at the $j$th row and $i$th column will be written $h^j_i$). Clearly if $(x,F) \in \mathcal{F}_O(M)$ and $h \in O(n)$, then $(x,F \cdot h) \in \mathcal{F}_O(M)$ as well. On Riemannian manifolds, orthonormal frames are often more convenient to work with than coordinate frames.

The frame bundle has a canonical construction called the \emph{solder form}. This is a vector-valued one-form $\theta \in C^\infty\Omega^1(\mathcal{F}_{GL}(M);\mathbb{R}^n)$ defined implicitly by 
\[
d\pi\left(v|_{(x,F)}\right) = \theta^j\left(v|_{(x,F)}\right)F_j, 
\]
where $\pi : \mathcal{F}_{GL}(M) \to M$ is the bundle projection and $d\pi : T\mathcal{F}_{GL}(M) \to TM$ is its tangent map. When $f$ is a smooth frame, the one-forms $\{f^*\theta^j\}_{j = 1}^n$ form a basis of the cotangent space $T_x^*M$ such that $(f^*\theta^j)(f_k) = \delta^j_k$.

A vector $v|_{(x,F)} \in T_{(x,F)}\mathcal{F}_{GL}(M)$ will be called \emph{vertical} if $\theta(v) = 0$. Since the group action of $GL(n)$ on the fibers of $\mathcal{F}_{GL}(M)$ is free, all vertical vectors are derivatives at time $t = 0$ of curves $t \mapsto  (x,F \cdot h_v(t))$, where $h_v: (-1,1) \to GL(n)$ is a smooth curve with $h_v(0) = I$. This means we can define a linear map $\eta: \ker\theta \subset T_{(x,F)}\mathcal{F}_{GL}(M) \to \mathfrak{gl}(n)$ by $v \mapsto \dot{h_v}(0)$. This is a linear isomorphism at each point, with inverse $\eta^{-1}(\dot{h}(0)) = \pdiff{}{t}|_{t = 0}(x,F \cdot h(t))$. When restricted to vertical vectors in $T_{(x,F)}\mathcal{F}_O(M)$, $\eta$ becomes $\mathfrak{so}(n)$-valued.

It can be shown that $\eta$ defined this way is smooth by taking a coordinate neighborhood $U \subset M$, which defines a smooth section $s: U \to \mathcal{F}_{GL}(U)$ by $s_i := \pdiff{}{x_i}$. Every point $(x,F) \in \mathcal{F}_{GL}(U)$ is then equal to $(x,s \cdot h)$ for some matrix $h \in GL(n)$, which means $\mathcal{F}_{GL}(U)$ has smooth coordinates $(x^i,h^j_k)$. When expressed in these coordinates, $\eta$ is equal to $h^{-1}dh$. Thus $\eta$ is called the (left-invariant) \emph{Maurer-Cartan form}, ubiquitous in the theory of the geometry of Lie groups and symmetric spaces \cite{Sharpe}.

One of the classical theorems in Riemannian geometry is the existence and uniqueness of a metric-compatible, torsion-free connection called the Levi-Civita connection. Similar reasoning \cite{Gardner} can be used to derive the existence and uniqueness of an $\mathfrak{so}(n)$-valued form $\omega \in \Omega^1(\mathcal{F}_O(M);\mathfrak{so}(n))$ such that, in $\mathcal{F}_O(M)$,
\begin{equation}\label{structureeqn1}d\theta^i = -\omega^i_j \wedge \theta^j.\end{equation}

This form $\omega$ encodes the Levi-Civita connection. In fact, if $\nabla$ is the usual Levi-Civita connection and $f: M \to \mathcal{F}_O(M)$ is a smooth orthonormal frame, then $f^*\omega^i_j(v) = \iprod{\nabla_vf_j,f_i}$. Furthermore, the curvature of this connection,
\begin{equation}\label{structureeqn2}
\Omega := d\omega + \frac{1}{2}[\omega,\omega],
\end{equation}
written in coordinates as $\Omega^i_j = d\omega^i_j +\omega^i_k \wedge \omega^k_j$, is equivalent to the Riemann curvature tensor in the sense that $f^*\Omega^i_j(X,Y) = \iprod{R_{X,Y}f_j,f_i}$, where $R_{X,Y}=\nabla_X\nabla_Y-\nabla_Y\nabla_X-\nabla_{[X,Y]}$.  More geometric identities, such as the Bianchi identities, can be derived from these \emph{structure equations}, but these two are sufficient for the purposes of this paper.

For convenience we will introduce a basis of $\mathfrak{so}(n)$ that is used throughout this paper. For $i \ne j$, let $w^i_j \in \mathfrak{so}(n)$ be defined as the matrix such that the entry at the $i$th row and $j$th column is equal to $-1$, while the entry at the $j$th row and $i$th column is equal to $1$, and all other entries are zero. Every element $A \in \mathfrak{so}(n)$ can therefore be written as $\sum_{i < j} A^j_i w^i_j$.

The last piece of background from Lie group theory that is needed is the \emph{adjoint action}. It comes from the conjugation action of $GL(n)$ on itself, which we will denote $\mathbf{Ad}(h)(k) := hkh^{-1}$. The adjoint action $\mathrm{Ad}: GL(n) \to \mathrm{Aut}(\mathfrak{gl}(n))$ is defined by $\mathrm{Ad}(h)(A) = hAh^{-1}$. Abstractly, $\mathrm{Ad}(h)(A)$ is the derivative of $\mathbf{Ad}(h)(k(t))$ at $t = 0$, with $k$ being any curve such that $k(0) = I, \dot{k}(0) = A$.

A $\mathfrak{gl}(n)$-valued form $\alpha$ on $\mathcal{F}_{GL}(M)$ is said to be \emph{tensorial} (or \emph{semi-basic}) if $v \lrcorner \alpha = 0$ for any vertical vector field $v$ (so $\alpha$ evaluates to zero on any multivector that has a vertical component) and $\alpha|_{(x,F \cdot h)} = \mathrm{Ad}(h^{-1})(\alpha|_{(x,F)})$ for any $h \in O(n)$. This expresses the idea that $\alpha$, in some sense, does not depend on the frame; if $\alpha$ is tensorial, then an endomorphism-valued form $\hat\alpha \in \Omega^k(M;\mathrm{End}(TM))$ could be defined so that $F_i \alpha^i_j|_{(x,F)} = \hat\alpha|_x(F_j)$.

The 2-form $\Omega$ defined above is tensorial, but the 1-form $\omega$ is not. In fact, if we were to calculate out $\omega(v)$ where $v$ is a vertical vector, then we get the same as $\eta(v)$. This can also be expressed by the gauge transformation law: if $f: M \to \mathcal{F}_O(M)$ is an orthonormal frame and $h: M \to O(n)$ is a change of basis, then 
\begin{equation} \label{gauge}
(f \cdot h)^*\omega = h^{-1}dh + \mathrm{Ad}(h^{-1})(f^*\omega).
\end{equation}
In other words, the Levi-Civita connection is not a tensor.

\begin{remark}\label{parallelizableremark}It must be noted that, in this section, we have treated $M$ as if there can exist a global frame $f: M \to \mathcal{F}_{GL}(M)$. In many cases, this is not possible. Manifolds that do support a global frame are called \emph{parallelizable}. However, every point in a manifold is contained in a neighborhood which is parallelizable. Since curvature is a local property, this means the parallelizable case is really the most interesting. Parallelizability is also different from being topologically trivial; for instance, spheres of dimension $0$, $1$, $3$, and $7$ and products thereof are all parallelizable.\end{remark}

\subsection{Meshes and Regge Metrics}\label{meshessection}

In classical differential geometry, the metrics and frames in question are always smooth. However, this paper is concerned with Regge metrics, which are only piecewise smooth with respect to a mesh. In this section we will precisely define much of the terminology of meshes and Regge metrics that is used throughout the rest of the paper.

$M$ will be assumed to be a smooth, oriented, polyhedral $n$-manifold (see the appendix for precise definitions). We will also assume that $M$ is parallelizable, with the idea in mind that more topologically complicated manifolds can be obtained by gluing together finitely many parallelizable ones; for instance a sphere can be obtained by gluing together two disks. $M$ is equipped with a countable \emph{mesh} $\{\Delta\}$ whose union is $M$, where each set $\Delta \subset M$ is the image of a closed $(n-d)$-dimensional convex polytope $\hat\Delta \subset \mathbb{R}^{n-d}$ under a smooth embedding for some $d \le n$. If $d = 0$, then the embedding must in addition be positively oriented.  We will abuse terminology and call each $\Delta$ a polytope, even though it is technically the image of a polytope under a smooth embedding.

The different types of polytope are distinguished by their codimension $d$ in $M$; polytopes of codimension 0 will be labeled $T$, polytopes of codimension 1 will be labeled $e$, and polytopes of codimension 2 will be labeled $p$. Polytopes of arbitrary codimension will simply be labeled $\Delta_d$, where $d$ is the codimension, so a polytope labeled $\Delta_n$ is a single point, a polytope labeled $\Delta_{n-1}$ is a line segment, and so on. The relative interior of a polytope $\Delta_d$, meaning the set of all points $x \in \Delta_d$ such that there exists an open set $U$ containing $x$ and $U \cap \Delta_d \cong \mathbb{R}^{n-d}$, will be denoted $\mathring{\Delta}_d$.

As a technical assumption we will need to assume that each top-dimensional polytope $T$ has a \emph{blow-up} which is a closed convex polytope in $\mathbb{R}^n$. This is at least true for simplices, and we believe it is probably true for general convex polytopes. See the appendix for more information on blow-ups and relevant citations.

The mesh must satisfy some axioms. Each face of a polytope in the mesh is also a polytope of the mesh. The intersection of two polytopes must be either a shared face of both polytopes or empty. Additionally, the mesh must respect the stratification structure of $M$. Being a polyhedral manifold, the boundary $\partial M$ can be decomposed into strata $S_d(M)$, which consists of those points $x \in \partial M$ which are contained in a submanifold of $\partial M$ that is of codimension $d$, but not one that is of codimension $d-1$. We will require that if the relative interior of a polytope $\Delta_d$ intersects $S_{d'}(M)$, then $\mathring{\Delta}_d \subseteq S_{d'}(M)$. This prevents pathological tangencies at the boundary, and it means that the closure of each stratum of $M$ inherits its own mesh decomposition.

A \emph{Regge metric} $g$ for the manifold $M$ and mesh $\{\Delta\}$ consists of a $C^2$ Riemannian metric $g^{\Delta}$ for each polytope $\Delta$, with the property that if $\Delta'$ is a face of $\Delta$ then $i_{\Delta'}^*g^{\Delta} = g^{\Delta'}$, where $i_{\Delta'} : \Delta' \hookrightarrow \Delta$ is the inclusion map. In other words, the tangential-tangential components of $g$ are continuous across any shared face. Ideally, one would like to remove the positive-definiteness restriction and work with arbitrary pseudo-Riemannian geometries.  We have attempted to avoid using the positive-definiteness whenever possible, but nonetheless this is still a very necessary assumption to make. In subsequent sections, $M$, $\{\Delta\}$, and $g$ are implicit.

We make frequent use of a ``wedge inner product'' between forms that take values in $\mathfrak{gl}(n)$. We can define a nondegenerate symmetric bilinear form $\iprod{\cdot,\cdot}$ on the vector space $\mathfrak{gl}(n) = \mathbb{R}^{n \times n}$ by setting $\iprod{A,B} := \mathrm{Tr}(AB)$. When restricted to $\mathfrak{so}(n)$, it is a negative-definite inner product. This means we can define a product 
\[
\iprod{\cdot\wedge\cdot}: (\mathfrak{gl}(n) \otimes \Lambda^k_x(M)) \otimes (\mathfrak{gl}(n) \otimes \Lambda^{n-k}_x(M)) \to \Lambda^n_x(M)
\]
by setting 
\[
\iprod{(A \otimes \alpha) \wedge (B \otimes \beta)} := \iprod{A,B}\alpha \wedge \beta
\]
and extending multilinearly. This is not a true inner product, but it is nondegenerate and bilinear, and it is completely independent of any metric structure, which makes it desirable for our use case. It also has some useful symmetries. One we will use often is that $\iprod{(\mathrm{Ad}(h)(A) \otimes \alpha) \wedge (B \otimes \beta)} = \iprod{(A \otimes \alpha) \wedge (\mathrm{Ad}(h^{-1})(B) \otimes \beta)}$ for any $h \in GL(n)$. When applying the adjoint action (or any other map $\mathfrak{gl}(n) \to \mathfrak{gl}(n)$) to a Lie algebra part of a Lie algebra valued form, we will abuse notation slightly by applying it to the whole form.

\section{Derivation of the Distributional Riemann Curvature}\label{distcurvsection}

The distributional Riemann curvature tensor associated to an orthonormal frame $f$, which we will denote by $f^*\Omega_{\mathrm{dist}}$, is a linear functional which associates a number to each smooth, compactly supported $\mathfrak{so}(n)$-valued $(n-2)$-form $\phi$ which vanishes when pulled back to $\partial M$. The frame $f$ is made up of $C^2$ frames on the interior of each codimension-0 polytope $T$, so $f = \bigsqcup_{T \subseteq M} f^T$ and each $f^T: \mathring{T} \to \mathcal{F}_O(\mathring{T})$ is a $C^2$ section. We will also take as a definition that $f^*\omega$ and $f^*[\omega,\omega]$ are the piecewise-$C^2$ forms defined as $f^*\omega|_{\mathring{T}} := {f^T}^* \omega$ and $f^*[\omega,\omega]|_{\mathring{T}} := {f^T}^* [\omega,\omega]$, where $\omega \in C^1\Omega^1(\mathcal{F}_O(\mathring{T});\mathfrak{so}(n))$ is the usual connection one-form.  The distributional exterior derivative of $f^*\omega$ is the linear functional defined by

\begin{equation} \label{distd}
\iprod{\iprod{ df^*\omega, \phi }} := \sum_{T \subseteq M} \int_{\mathring{T}} \iprod{ {f^T}^*\omega \wedge d\phi }
\end{equation}
for all $\phi \in C^\infty_c\Omega^{n-2}(M;\mathfrak{so}(n))$ that vanish when pulled back to $\partial M$.

Per the discussion above, our definition of $f^*\Omega_\mathrm{dist}$ is
\begin{equation}\label{distcurvdef}\iprod{\iprod{f^*\Omega_\mathrm{dist},\phi}} := \iprod{\iprod{df^*\omega + \frac{1}{2}f^*[\omega,\omega],\phi}} = \sum_{T \subseteq M} \int_{\mathring{T}} \iprod{{f^T}^*\omega \wedge d\phi} + \frac{1}{2}\iprod{{f^T}^*[\omega,\omega] \wedge \phi}.\end{equation} 

What we aim to do is find conditions on the frame $f$ such that the right-hand side of (\ref{distcurvdef}) can be efficiently computed, is bounded by a multiple of the supremum norm of $\phi$, and transforms like a tensor.

\subsection{Conditions on Compatible Frames} \label{compatframesection}
An orthonormal frame $f = \bigsqcup_{T \subseteq M} f^T$ will be called \emph{compatible} if it has some desirable properties that make $\iprod{\iprod{f^*\Omega_\mathrm{dist},\phi}}$ both correct from a geometrical standpoint and practical from a computational standpoint. The conditions that a compatible frame must satisfy are fairly technical and may seem arbitrary, so before stating them, we will first provide some motivation.

First, we want the individual vector fields in our frame to be ``parallel'' across codimension-1 polytopes, in some sense. For each interior codimension-1 polytope $\mathring{e} \subset \mathring{M}$, we can define a frame $E_e$ which is orthonormal in the metric $i_e^*g$; this metric is well-defined since $g$ has single-valued tangential-tangential components on $e$. If $e = T \cap T'$, then $E_e$ can be extended to two orthonormal frames $E_e^T$ and $E_e^{T'}$ by appending outward-facing normal vectors to $e$ for $g^T$ and $g^{T'}$ respectively, which we denote by $\vec{n}$ and $\vec{n}'$. Assume that for each $T$, $f^T$ can be continuously extended to $\mathring{e}$ by the $C^1$ section $f^T|_{\mathring{e}}: \mathring{e} \to \mathcal{F}_O(T)|_{\mathring{e}}$, and let $\mu_e^T: \mathring{e} \to O(n)$ be a map such that 
\[
E_e^T \cdot \mu_e^T = f^T|_{\mathring{e}}
\]
and likewise $E_e^{T'} \cdot \mu_e^{T'} = f^{T'}|_{\mathring{e}}$. The condition on codimension-1 faces is that 
\[
\mu_e^T = \begin{bmatrix}I & 0 \\ 0 & -1\end{bmatrix}\mu_e^{T'}.
\]
That is, the tangential components of each $f_i$ are continuous and the normal component of $f_i$ is continuous if one of the normal vectors is negated, so $\iprod{f_i,\vec{n}}_T = \iprod{f_i',-\vec{n}'}_{T'}$.

While somewhat arbitrary, this notion of ``parallelism'' is supported by the fact that a piecewise-smooth geodesic, defined as a locally energy-minimizing curve, must have a velocity vector that satisfies the same condition we have placed on the $f_i$'s.

Another, possibly deeper reason for this to be true, is that we need some kind of frame-independent coupling across codimension-1 polytopes for the distributional curvature to be tensorial. Suppose a piecewise-smooth frame $f$ is compatible and can be continuously extended to each codimension-1 boundary component $\mathring{e} \subset \partial T$ for each $T \subseteq M$. Then if $\phi$ has support in a small neighborhood of a point $x_0 \in \mathring{e}$, where $e = T_1 \cap T_2$ is an interface between two polytopes, Stokes' theorem gives us
\begin{align*}
\iprod{\iprod{f^*\Omega_\mathrm{dist},\phi}} 
&= \iprod{\iprod{df^*\omega + \frac{1}{2}f^*[\omega,\omega],\phi}} \\
&= \sum_{i = 1,2}\int_{\mathring{T_i}} \left(\iprod{{f^{T_i}}^*\omega \wedge d\phi} + \frac{1}{2}\iprod{{f^{T_i}}^*[\omega,\omega] \wedge \phi}\right)\\
&= \sum_{i = 1,2} \int_{\mathring{T_i}} \iprod{{f^{T_i}}^*(d\omega + \frac{1}{2}[\omega,\omega]) \wedge \phi} - \int_{\mathring{e}} \iprod{[\![f^*\omega]\!]\wedge \phi} \\
&= \sum_{i = 1,2} \int_{\mathring{T_i}} \iprod{{f^{T_i}}^*\Omega \wedge \phi} - \int_{\mathring{e}} \iprod{[\![f^*\omega]\!]\wedge \phi},
\end{align*}
where $[\![f^*\omega]\!]$ denotes the jump in $f^*\omega$ across $e$. This expression does not depend on any derivatives of $\phi$, so the domain of $f^*\Omega_{\mathrm{dist}}$ can be formally extended to include piecewise-smooth forms (with the same support) by setting 
\[
\iprod{\iprod{f^*\Omega_{\mathrm{dist}}, \phi}} = \sum_{i = 1,2}\int_{\mathring{T_i}} \iprod{{f^{T_i}}^*\Omega \wedge \phi } - \int_{\mathring{e}} [\![\iprod{f^*\omega \wedge \phi}]\!].
\]
However, keeping $\phi$ continuous, we could then apply differing transformations $h^{T_i}: T_i \to O(n)$ on either side of $e$, which we will collectively call $h$, and obtain another piecewise-smooth frame that we can evaluate the distributional curvature in. Since $\Omega$ is tensorial and $\omega$ obeys the gauge transformation law~\eqref{gauge}, we would get
\begin{align*}
&\iprod{\iprod{(f \cdot h)^*\Omega_\mathrm{dist},\phi}} \\
&= \sum_{i = 1,2} \int_{\mathring{T_i}} \iprod{\mathrm{Ad}({h^{T_i}}^{-1})({f^{T_i}}^*\Omega) \wedge\phi} - \int_{\mathring{e}} \iprod{[\![\mathrm{Ad}(h^{-1})(f^*\omega) + h^{-1}dh]\!] \wedge \phi}\\
&= \sum_{i = 1,2} \int_{\mathring{T_i}}\iprod{{f^{T_i}}^*\Omega \wedge\mathrm{Ad}({h^{T_i}})(\phi)} - \int_{\mathring{e}}[\![\iprod{f^*\omega \wedge \mathrm{Ad}(h)(\phi)}]\!] + \iprod{[\![h^{-1}dh]\!] \wedge \phi}\\
&= \iprod{\iprod{f^*\Omega_{\mathrm{dist}},\mathrm{Ad}(h)(\phi)}} - \sum_{\mathring{e} \subset \mathring{M}} \int_{\mathring{e}} \iprod{[\![h^{-1}dh]\!] \wedge \phi}.
\end{align*}

If $f^*\Omega$ were a continuous $\mathfrak{so}(n)$-valued form, we would get\\ $\iprod{\mathrm{Ad}(h^{-1})(f^*\Omega) \wedge \phi} = \iprod{f^*\Omega \wedge \mathrm{Ad}(h)(\phi)}$. Guided by this, we define
\begin{equation}
\iprod{\iprod{\mathrm{Ad}(h^{-1})(f^*\Omega_{\mathrm{dist}}),\phi}} := \iprod{\iprod{f^*\Omega_{\mathrm{dist}},\mathrm{Ad}(h)(\phi)}}.
\end{equation}

As the Riemann curvature should be tensorial, it should always be true that changing the frame by $h$ results in an adjoint action by $h^{-1}$ on $f^*\Omega_{\mathrm{dist}}$. For this to be true for continuous $\phi$ and discontinuous $h$, it must be the case that $[\![h^{-1}dh]\!] = 0$ along every codimension-1 face $e = T \cap T'$, so on $e$, $h^T = C_eh^{T'}$ for some constant matrix $C_e$. We will restrict ourselves to the case $C_e = I$, since otherwise it would not be possible to modify the transformation $h$ so that it is the identity outside of a small neighborhood of $x_0$ but retains the jump condition, losing locality. What this means is that, if any particular frame $f$ is asserted to be compatible, then it is reasonable to assert that in a small enough neighborhood $U$ of a point $x \in \mathring{e}$, the set of compatible frames on $U$ (meaning restrictions of compatible frames to $U$) must be contained in $\{(f|_U \cdot h): h\; \text{is piecewise-smooth and continuous}\}$. Therefore, piecewise-smooth frames are usually not compatible, and the set of piecewise-smooth compatible frames must all have the same jump conditions along codimension-1 polytopes---the matrix $\mu_e^T(\mu_e^{T'})^{-1}$ cannot depend on $f$. The choice $\mu_e^T(\mu_e^{T'})^{-1} = \begin{bmatrix}I & 0 \\ 0 & -1\end{bmatrix}$ is the simplest, and it is consistent with the case of continuous metrics and frames.

Enforcing the constraint $\mu_e^T = \begin{bmatrix}I & 0 \\ 0 & -1\end{bmatrix}\mu_e^{T'}$ on a frame precludes the possibility of it being continuous on the boundary of each polytope. For each codimension-2 polytope $p \subset M$, let $E_p$ be a frame which is orthonormal with respect to $i_p^*g$, and for each pair of codimension-0 and codimension-1 polytopes $T,e$ such that $p \subset e \subset T$, let $E_{p,e}^T$ be the extension of $E_p$ to an orthonormal frame by appending the normal vector $\vec{\nu}$ which is orthogonal to $p$ and points into $e$ and the normal vector $\vec{n}$ which is orthogonal to $e$ and makes $E_{p,e}^T$ right-hand oriented.  If $p = e \cap e'$ and $e,e' \subset T$, then at each point $x \in p$,

\begin{equation} \label{edge2edge}
E_{p,e}^T = E_{p,e'}^T \cdot \begin{bmatrix}I & 0 & 0 \\
0 & \cos(\pm\theta_p^T) & -\sin(\pm\theta_p^T) \\
0 & \sin(\pm\theta_p^T) & \cos(\pm\theta_p^T)\end{bmatrix} = E_{p,e'}^T\cdot\exp\left(\pm\theta_p^Tw^{n-1}_n\right),
\end{equation}
where $\theta_p^T$ is the dihedral angle between $e$ and $e'$ at $x$.
Additionally, there exists an $O(n)$-valued matrix $A_{p,e}^T$ such that 
\[
E_e^T = E_{p,e}^T \cdot A_{p,e}^T.
\]
Note that since the final entries of $E_e^T$ and $E_{p,e}^T$ are associated to normal vectors to $e$ in the metric $g^T$, $A_{p,e}^T = \begin{bmatrix}A_{p,e} & 0 \\ 0 & \pm 1\end{bmatrix}$, where $A_{p,e}$ does not depend on the polytope $T$. Because $E_e^{T'}$ has a normal vector that points in a different direction to that of $E_e^T$, but $E_{p,e}^{T'}$ does not, we also have $A_{p,e}^T = \begin{bmatrix}I & 0 \\ 0 & -1\end{bmatrix}A_{p,e}^{T'} = A_{p,e}^{T'}\begin{bmatrix}I & 0 \\ 0 & -1\end{bmatrix}$, and therefore
\[
A_{p,e}^T \mu_e^T = A_{p,e}^{T'} \mu_e^{T'}.
\]

Now suppose that for any Regge metric $g$, there exists a frame $f$ that is continuous on each codimension-0 polytope $T$ and satisfies the compatibility condition $\mu_e^T = \begin{bmatrix}I & 0 \\ 0 & -1\end{bmatrix}\mu_e^{T'}$. We will produce a contradiction. Let $p$ be a codimension-2 polytope which is completely surrounded by codimension-0 polytopes $T_1,\dots,T_k$, and let $e_i = T_i \cap T_{i+1}$ for $i=1,2,\dots,k-1$ and $e_0=e_k = T_k \cap T_1$. Without loss of generality, also assume that the ordering is chosen so that for each $i=1,2,3,\dots,k$, the signs in~\eqref{edge2edge} are positive when $e=e_{i-1}$, $e'=e_i$, and $T=T_i$.

Then 
\[
f^{T_i} = E_{p,e_{i-1}}^{T_i} \cdot ( A_{p,e_{i-1}}^{T_i}\mu_{e_{i-1}}^{T_i})  = E_{p,e_i}^{T_i} \cdot \left( \exp\left(\theta_p^{T_i}w^{n-1}_n\right) A_{p,e_{i-1}}^{T_i}\mu_{e_{i-1}}^{T_i}\right),
\]
but also $f^{T_i} = E_{p,e_i}^{T_i} \cdot (A_{p,e_i}^{T_i}\mu_{e_i}^{T_i})$.  Since the group action is free and $A_{p,e_{i-1}}^{T_i}\mu_{e_{i-1}}^{T_i} = A_{p,e_{i-1}}^{T_{i-1}}\mu_{e_{i-1}}^{T_{i-1}}$, this implies 

\begin{equation}\label{continuousframejump} \exp\left(\theta_p^{T_i}w^{n-1}_n\right)A_{p,e_{i-1}}^{T_{i-1}}\mu_{e_{i-1}}^{T_{i-1}}  = A_{p,e_i}^{T_i}\mu_{e_i}^{T_i}.\end{equation}

Let $\Pi_{i-1}^i$ be the linear transformation sending $f^{T_{i-1}}$ to $f^{T_i}$, meaning $\Pi_{i-1}^i(f^{T_{i-1}}_j) = f^{T_i}_j$ for each $j$. Clearly, $\Pi_k^1 \circ \Pi_{k-1}^k \circ \dots \circ \Pi_1^2 = I$. Let us express the transformation $\Pi_{i-1}^i$ in the bases $E_{p,e_{i-1}}^{T_{i-1}}$ and $E_{p,e_{i}}^{T_{i}}$:
\[
\![\Pi_{i-1}^{i}]_{E_{p,e_{i-1}}^{T_{i-1}}}^{E_{p,e_{i}}^{T_{i}}} = A_{p,e_{i}}^{T_{i}}\mu_{e_{i}}^{T_{i}}(A_{p,e_{i-1}}^{T_{i-1}}\mu_{e_{i-1}}^{T_{i-1}})^{-1} = \exp\left(\theta_p^{T_i}w^{n-1}_n\right).
\]

So we should get that 
\[
I = [\Pi_k^1 \circ \dots \circ \Pi_1^2]_{E_{p,e_1}^{T_1}}^{E_{p,e_1}^{T_1}} = \exp\left(\left(\sum_{i = 1}^k \theta_p^{T_i}\right)w^{n-1}_n\right),
\] 
which implies $\sum_{i = 1}^k \theta_p^{T_i} = 2m\pi$ for some integer $m$. However, in general, this sum can take any positive value if the metric is discontinuous at $p$. Therefore the frame $f^T$ cannot always be continuous at $p$. This is the origin of the angle defect. In order to control the discontinuity of $f^T$ as much as possible, we will restrict it to only rotate at a constant speed, and only in the plane orthogonal to $p$.

\paragraph{Summary of objects introduced.}

The preceding paragraphs introduce some objects that are used throughout the rest of the paper. They are collected here for convenience.

\begin{tabular}{c|p{12cm}}
Notation & Definition \\
\hline
$E_p$ & An arbitrary orthonormal frame on the codimension-2 polytope $p$. Entries are called $\tau_1,\dots,\tau_{n-2}$.\\
$E_e$ & An arbitrary orthonormal frame on the codimension-1 polytope $e$. Entries are called $E_1,\dots,E_{n-1}$. \\
$\vec{n}^T$ & The outward-pointing $g^T$-normal vector to a face $e \subset T$. Usually called just $\vec{n}$ when $T$ is implicitly known.\\
$\vec{\nu}$ & The inward-pointing normal vector to a face $p \subset e$.\\
$E_{p,e}$ & The $g^e$-orthonormal frame defined on $p$ with entries $(\tau_1,\dots,\tau_{n-2},\vec{\nu})$\\
$A_{p,e}$ & The map $p \to O(n-1)$ such that $E_e = E_{p,e} \cdot A_{p,e}$\\
$E_e^T$ & The $g^T$-orthonormal frame defined on $e$ with entries $(E_1,\dots,E_{n-1},\vec{n}^T)$.\\
$E_{p,e}^T$ & The $g^T$-orthonormal frame defined on $p$ with entries $(\tau_1,\dots,\tau_{n-2},\vec{\nu},\pm\vec{n}^T)$, where the sign on $\vec{n}^T$ is chosen so that $E_{p,e}^T$ is positively oriented in $M$.\\
$A_{p,e}^T$ & The map $p \to O(n)$ such that $E_e^T = E_{p,e}^T \cdot A_{p,e}^T$.\\
$\mu_e^T$ & The map $\mathring{e} \to O(n)$ such that $f^T|_{\mathring{e}} = E_e^T \cdot \mu_e^T$.\\
$\theta_p^T$ & The map $p \to \mathbb{R}$ which measures the interior angle between the two faces $e,e' \subset T$ whose intersection is equal to $p$.\\
\end{tabular}\\

To ensure that we are still able to apply integration by parts, despite the fact that the frame is discontinuous at $p$, we will follow the strategy outlined in~\cite[p. 2]{BERCHENKOKOGAN2025100529} and require that $f^T$ has some smoothness and continuity when pulled back to the \emph{blow-up} $B_T$ of $T$.  The blow-up $B_T$ of $T$ is essentially a polytope that has one codimension-1 face for each codimension-$d$ face of $T$ with $d \ge 1$, with tangencies related to inclusion relations between the original faces. There is a corresponding \emph{blow-down} map $\Phi^T: B_T \to T$ which restricts to a diffeomorphism $\Phi^T|_{\mathring{B}_T}: \mathring{B}_T \to \mathring{T}$. The blow-up has an exceptional set which we will call $E_{B_T}$, consisting of the faces of $B_T$ of codimension $\ge 2$, and it can be safely ignored for the purpose of Riemann integration. Some information on blow-ups of manifolds can be found in the appendix.

With all this in mind, we arrive at the definition of a compatible frame:
\begin{definition}\label{compatibilitydef}
A frame $f = \bigsqcup_{T \subseteq M} f^T$, where $f^T: \mathring{T} \to \mathcal{F}_O(T)|_{\mathring{T}}$ is a $C^2$ orthonormal frame for the metric $g^T$, is \emph{compatible} if:
\begin{enumerate}
\item{For each $T \subseteq M$, there exists a blow-up $B_T$ (equipped with a blow-down map $\Phi^T : B_T \to T$) and a Lipschitz continuous map $F^T: B_T \to \mathcal{F}_O(T)$ such that $F^T|_{B_T \backslash E_{B_T}} \in C^2(B_T \backslash E_{B_T}; \mathcal{F}_O(T))$, $f^T \circ \Phi^T|_{\mathring{B}_T} = F^T|_{\mathring{B}_T}$, and $\pi \circ F^T = \Phi^T$. We could also say that the following diagram is commutative (wherever the maps are defined):
\begin{center}
\begin{tikzcd}[column sep=large, row sep = large]
    B_T \arrow[r,"F^T"] \arrow[rd, "\Phi^T"] & \mathcal{F}_O(T) \arrow[d,"\pi", shift left=1ex] \\
    & T \arrow[u,"f^T"] 
\end{tikzcd} 
\end{center}
This ensures that for each $e \subset T$, there is a $C^2$ section $f^T|_{\mathring{e}}: \mathring{e} \to \mathcal{F}_O(T)|_{\mathring{e}}$ which continuously extends $f^T$.}

\item{For each $e \subset T \subseteq M$, let $\mu_e^T: \mathring{e} \to O(n)$ be the matrix such that $E_e^T \cdot \mu_e^T = f^T|_{\mathring{e}}$, and suppose $e = T \cap T'$. Then $\mu_e^{T'} = \begin{bmatrix}I & 0 \\ 0 & -1\end{bmatrix}\mu_e^T$.}
\item{For each $p \subset T \subseteq M$, let $e \subset T$ be the unique face meeting $p$ such that the frame $E_{p,e}^T$ has an inward-pointing normal vector as its last entry. Then there exists an orientation-preserving embedding $\psi^T_p: [0,1] \times \mathring{p}\to \overline{{\Phi^T}^{-1}(\mathring{p})}$, where the orientation in $p$ is induced by the orientation on $e \subset T$, and the orientation on ${\Phi^T}^{-1}(\mathring{p})$ is induced from $B_T$, such that $\psi^T_p(\{0\} \times \mathring{p}) \subset \overline{{\Phi^T}^{-1}(\mathring{e})}$ and $\Phi^T \circ \psi^T_p(s,x) = x$. Additionally, there exists a continuous function $r_{p,e}^T: p \to \mathbb{R}$ such that
\[
F^T(\psi^T_p(s,x)) = E_{p,e}^T(x) \cdot \left(\exp\left(s~r_{p,e}^T(x)w^{n-1}_n\right)A_{p,e}^T(x)\mu_e^T(x)\right)
\]
for all $x \in \mathring{p}$ and all $s \in [0,1]$.  In other words, the multi-valuedness of $f^T$ at codimension-2 faces is controlled so that $f^T$ is continuous when expressed in cylindrical coordinates around $x \in \mathring{p}$ and rotates in the plane orthogonal to $p$ at a rate depending only on $x$. There is no restriction on behavior near faces of higher codimension. 
}
\item{There exists a smooth metric $g_0$ and a smooth $g_0$-orthonormal frame $f_0$ such that there is a continuous homotopy of Regge metrics $g(t)$ with $g(1) = g$ and $g(0) = g_0$, and a homotopy of compatible frames $f(t)$ such that $f(0) = f_0$, $f(1) = f$, and $f(t)$ is $g(t)$-orthonormal and satisfies conditions 1, 2, and 3. The map $(t,x) \mapsto F^T(t)(x)$ also must vary continuously as a map $[0,1] \times B_T \to \mathcal{F}_{GL}(T)$.}
\end{enumerate}

\end{definition}

As we mentioned above, condition 1 makes it possible to integrate by parts even though the frame may have discontinuities on polytopes of codimension 2 or greater.
The fact that $F^T$ is Lipschitz and $C^2$ on the set $B_T \backslash E_{B_T}$ is necessary for the identity $d{F^T}^*\omega = {F^T}^*d\omega$ to hold and for ${F^T}^*\omega$ to be bounded and continuously extendable to $B_T \backslash E_{B_T}$. These are all necessary conditions for Stokes' Theorem to hold for the form ${F^T}^*(\omega \wedge \pi^*\phi)$ on $B_T$. Conditions 2 and 3 are necessary for the frame $f$ to have some semblance of continuity and for the distributional curvature to transform like a tensor, as mentioned previously, and condition 4 states that compatible frames can be ``smoothed out'' without introducing singularities.

\paragraph{Intuition behind the blow-up.}

To illustrate how blow-ups can be used to facilitate integration by parts, consider the following example.  On the quarter-disk $Q = \{(x,y) \in \mathbb{R}^2 \mid x,y > 0, \, x^2+y^2 < 1\}$, let $\alpha = d\theta$, where $\theta = \arctan(y/x)$.  Since $d\alpha=0$, $\int_Q d\alpha=0$.  However, the integrals of $\alpha$ over the three one-dimensional faces of $Q$ do not sum to zero: $\int_{(0,1) \times \{0\}} \alpha = \int_{\{0\} \times (0,1)} \alpha = 0$ and $\int_{\partial Q \cap S^1} \alpha = \frac{\pi}{2}$.  This discrepancy can be attributed to the fact that $\alpha$ is discontinuous at the origin.  To fix the problem, one can instead carry out the integrals in the polar coordinate domain $(r,\theta) \in (0,1) \times (0,\frac{\pi}{2})$, whose boundary has four (as opposed to three) one-dimensional faces.  This time, the integrals of $\alpha$ over the four faces of the boundary sum to zero since the integrals over $\{0\} \times (0,\frac{\pi}{2})$ and $\{1\} \times (0,\frac{\pi}{2})$ cancel one another.

The polar coordinate domain $(0,1) \times (0,\frac{\pi}{2})$ plays a similar role to the blow-up $B_T$ appearing in condition 1, and the map $(r,\theta) \mapsto (r\cos\theta,r\sin\theta)$ that sends this domain to the quarter-disk $Q$ plays a similar role to the blow-down map $\Phi^T$ appearing in condition 1.  This example therefore illustrates how the blow-up $B_T$ facilitates integration-by-parts calculations. 

Having said all of this, one still might wonder if the use of $B_T$ could be avoided by performing integration by parts on $\overline{f^T(\mathring{T})}$ instead.  After all, we are allowing $f^T$ to have discontinuities on codimension-2 faces of $T$, so it is reasonable to expect the boundary of $\overline{f^T(\mathring{T})}$ to have ``extra'' codimension-1 faces much like $B_T$ does.
Unfortunately $\overline{f^T(\mathring{T})}$ may fail to be a Whitney manifold. For instance, this could happen if $r_{p,e}^T$ vanishes on a disconnected set which has nonzero measure in $p$. In a situation like this, integration by parts on $\overline{f^T(\mathring{T})}$ is infeasible. However, by pulling back forms by $F^T$ to the manifold $B_T$, which is regular enough to permit an integration by parts formula, we regain this ability even in the smooth category. Blow-ups are also a convenient tool for constructing compatible frames.

\begin{remark}We have ignored the question of whether the piecewise-smooth one-form $f^*\omega$ defined by ${f^*\omega}^i_j|_{\mathring{T}} = \iprod{\nabla^Tf_j^T,f_i^T}_T$ on $T$ is ``really" the connection form associated to the compatible frame $f$, despite the fact that $f$ is discontinuous. Here is one way to answer that question:  Since the pullback of $f^*\theta^i$ to every element interface $e = T \cap T'$ is single-valued, the distributional exterior derivative of $f^*\theta^i$ is simply its elementwise exterior derivative.  Therefore the structure equations $df^*\theta^i = -f^*\omega^i_j \wedge f^*\theta^j$ hold in a distributional sense.

\end{remark}

\subsection{The Integration by Parts Step} \label{integrationbypartssection}
To perform integration by parts, we proceed in each codimension-0 polytope $T$ by lifting the integrals in (\ref{distcurvdef}) to $B_T$ and expanding out the inner products in coordinates. For now, we will suppress the ${}^T$ superscripts and focus on one term of the integral:
\begin{align*}
\int_{\mathring{T}} \iprod{f^*\omega\wedge d\phi} &= \int_{\Phi(\mathring{B}_T)} f^*\omega^i_j \wedge d\phi^j_i \\
&= \int_{\mathring{B}_T} \Phi^*f^*\omega^i_j \wedge \Phi^*d\phi^j_i. 
\end{align*}
Next we apply the facts that $f \circ \Phi = F$ on $\mathring{B}_T$, so $\Phi^*f^* = F^*$, and $\Phi = \pi \circ F$, so $\Phi^* = F^* \pi^*$. This yields 
\[
\int_{\mathring{B}_T} \Phi^*f^*\omega^i_j \wedge \Phi^*d\phi^i_j= \int_{\mathring{B}_T} F^*(\omega^i_j \wedge \pi^*d\phi^j_i).
\]
The form $F^*(\omega^i_j \wedge \pi^*\phi^j_i)$ is $C^1$ and bounded on $\mathring{B}_T$, and summable on $\partial B_T \backslash E_{B_T}$, and $F^*d(\omega^i_j \wedge \pi^*\phi^j_i) = dF^*(\omega^i_j \wedge \pi^*\phi^j_i)$ is summable on $\mathring{B}_T$. Therefore the integration by parts theorem for Whitney manifolds~\cite[Theorem 18A]{whitney} applies:
\[
\int_{\mathring{B}_T} F^*(\omega^i_j \wedge \pi^*d\phi^j_i) = \int_{\mathring{B}_T}F^*(d\omega^i_j \wedge \pi^*\phi^j_i) - \int_{\partial B_T \backslash E_{B_T}} F^*(\omega^i_j \wedge \pi^*\phi^j_i).
\]
The first term can be pulled back down to $\mathring{T}$, and the second term can be split up into integrals over the faces of $B_T$ from different strata:
\begin{align*}
&\int_{\mathring{B}_T}F^*(d\omega^i_j \wedge \pi^*\phi^j_i) - \int_{\partial B_T \backslash E_{B_T}} F^*(\omega^i_j \wedge \pi^*\phi^j_i) \\
&= \int_{\mathring{T}}f^*(d\omega^i_j) \wedge \phi^j_i - \sum_{d = 1}^n\int_{\Phi^{-1}(S_d(T))}F^*(\omega^i_j \wedge \pi^*\phi^j_i).
\end{align*}
For $d > 2$, the integrals over $\Phi^{-1}(\mathring{\Delta}_d)$ vanish for any codimension-$d$ face $\mathring{\Delta}_d \subset S_d(T)$. Indeed, $\phi^j_i$ is an $(n-2)$-form, so its trace $i_{\mathring{\Delta}_d}^* \phi^j_i$ vanishes on any $\mathring{\Delta}_d$ with $d>2$.  Since $i_{\Phi^{-1}(\mathring{\Delta}_d)}^*F^*\pi^* = i_{\Phi^{-1}(\mathring{\Delta}_d)}^*\Phi^* = \Phi^*i_{\mathring{\Delta}_d}^*$, we have $i_{\Phi^{-1}(\mathring{\Delta}_d)}^*F^*\pi^*\phi^i_j = \Phi^*i_{\mathring{\Delta}_d}^*\phi^j_i = 0$.
Therefore the integral of $F^*(\omega^i_j \wedge \pi^*\phi^j_i)$ over $\Phi^{-1}(S_d(T)) = \Phi^{-1}(\bigcup \mathring{\Delta}_d)$ is equal to zero.

Additionally, since there is a smooth (in each component) section $f|_{S_1(T)}$ of $\mathcal{F}_O(T)|_{S_1(T)}$ that continuously extends $f$ (by condition 1), 
\[
\int_{\Phi^{-1}(S_1(T))} F^*(\omega^i_j \wedge \pi^*\phi^j_i) = \int_{S_1(T)}f|_{S_1(T)}^*(\omega^i_j \wedge \pi^*\phi^j_i ) = \int_{S_1(T)} f|_{S_1(T)}^*\omega^i_j \wedge \phi^j_i.
\]
Here, we have implicitly used the fact that $\Phi|_{\Phi^{-1}(S_1(T))} : \Phi^{-1}(S_1(T)) \to S_1(T)$ is a diffeomorphism on each component; see Appendix~\ref{appendix}.

Plugging this back into (\ref{distcurvdef}), we get
\begin{align}
\iprod{\iprod{f^*\Omega_{\mathrm{dist}},\phi}} &= \sum_{T \subseteq M}\int_{\mathring{T}} {f^T}^*(d\omega^i_j + \omega^i_k \wedge \omega^k_j )\wedge\phi^j_i - \int_{S_1(T)} f^T|_{S_1(T)}^*\omega^i_j \wedge \phi^j_i \label{distcurvpart0}\\
&\quad\quad - \int_{{\Phi^{T}}^{-1}(S_2(T))} {F^T}^*(\omega^i_j \wedge \pi^*\phi^j_i) \notag \\
&= \sum_{T \subseteq M}\int_{\mathring{T}} \iprod{{f^T}^*(d\omega + \frac{1}{2}[\omega,\omega]) \wedge \phi} - \sum_{\mathring{e} \subset \mathring{M}}\int_{\mathring{e}} \iprod{[\![f^T|_{\mathring{e}}^*\omega]\!] \wedge \phi} \label{distcurvpart1}\\
&\quad\quad - \sum_{\mathring{e} \subset \mathring{M}} \sum_{T \supset p} \int_{{\Phi^T}^{-1}(\mathring{p})} {F^T}^* \iprod{\omega \wedge \pi^*\phi}. \notag
\end{align}

For the integrals over codimension-1 polytopes, let $\Two^T_e \in \Omega^1(e,\mathfrak{so}(n))$ be the form defined by $(\Two^T_e)^i_j = 0$ for all $1 \le i,j \le n-1$ and $(\Two^T_e)^n_i = \iprod{\nabla (E_e^T)_i,\vec{n}}$ for $1 \le i \le n$. We will use it to represent the second fundamental form of $e$. Also let $\tilde{\omega}_e^T = (E_e^T)^*\omega$, that is, the full connection form for the frame $E_e^T$. Then the gauge transformation law~\eqref{gauge} tells us that $f^T|_{\mathring{e}}^*\omega = \mathrm{Ad}((\mu_e^T)^{-1})(\tilde \omega^T_e) + (\mu_e^T)^{-1}d\mu_e^T$.

Since the first $n-1$ vectors in $E_e^T$ do not depend on $T$, there exists an $\mathfrak{so}(n-1)$-valued form $J_e$ that does not depend on $T$ so that $\tilde\omega^T_e - \Two^T_e = \begin{bmatrix}J_e & 0 \\ 0 & 0\end{bmatrix}$. For interior edges $e = T \cap T'$, $(\mu_e^T)^{-1}d\mu_e^T = (\mu_e^{T'})^{-1}d\mu_e^{T'}$, and the following equation gives us the jump in $\omega$:
\begin{align*}
[\![f^T|_{\mathring{e}}^* \omega]\!] &= \mathrm{Ad}((\mu_e^T)^{-1})(\tilde\omega^T_e) - \mathrm{Ad}((\mu_e^{T'})^{-1})(\tilde\omega^{T'}_e)\\
& = \mathrm{Ad}((\mu_e^T)^{-1})\left(\tilde{\omega}^T_e - \mathrm{Ad}\left(\begin{bmatrix}I & 0 \\ 0 & -1\end{bmatrix}\right)(\tilde{\omega}^{T'}_e)\right) \\
&= \mathrm{Ad}((\mu_e^T)^{-1})(\Two_e^T + \Two_e^{T'}).
\end{align*}
This could still be interpreted as a jump in second fundamental form, since the two fundamental forms are being evaluated with normal vectors that point in opposite directions. We will therefore denote $[\![\Two_e]\!] := \Two_e^T + \Two_e^{T'}$ when $e = T \cap T'$. The integrals over codimension-1 polytopes then read
\begin{equation}\label{curvjump}
\int_{\mathring{e}}\iprod{[\![f^{T}|_{\mathring{e}}^*\omega]\!]\wedge \phi} = \int_{\mathring{e}}\iprod{\mathrm{Ad}((\mu_e^T)^{-1})([\![\Two_e]\!]) \wedge \phi}.
\end{equation}

While this expression involves a codimension-0 polytope $T$ of which $e$ is a face, the expression is not actually dependent on $T$, because both the sign of the integrand and the orientation of the integral change when $T$ and $T'$ are interchanged. The only caveat is that the integral must be evaluated using the induced orientation of $e$ from $T$. 

Note also that we have made key use of the fact that $i_{\mathring{e}}^*\phi = 0$ for $e \subset \partial M$, so there is no issue with the boundary terms here. See Section \ref{gengaussbonnetsection} for some discussion on how things change if $\phi$ is allowed to not pull back to zero on $\partial M$.

\subsection{Derivation of the Angle Defect}\label{angledefectsection}

Expanding out the last term on the right of \ref{distcurvpart1} takes more work. First, we will use condition 3 on $f$ to change the integrals over ${\Phi^T}^{-1}(\mathring{p})$ to integrals over $(0,1) \times \mathring{p}$:
\[\int_{{\Phi^T}^{-1}(\mathring{p})}{F^T}^*(\omega^i_j \wedge \pi^*\phi^j_i) = \int_{(0,1) \times \mathring{p}}{\psi_p^T}^*{F^T}^*(\omega^i_j\wedge\pi^*\phi^j_i).\]

The form ${\psi_p^T}^*{F^T}^*\pi^*\phi^j_i = {\psi_p^T}^*{\Phi^T}^*\phi^j_i$ is basic for the fiber bundle $q: [0,1] \times p \to p$, i.e. ${\psi_p^T}^*{\Phi^T}^*\phi^j_i = q^*i_p^*\phi^j_i$. Then Fubini's theorem can be applied to this integral:
\[\int_{(0,1) \times \mathring{p}}{\psi_p^T}^*{F^T}^*(\omega^i_j\wedge\pi^*\phi^j_i) = \int_{\mathring{p}}\left(\int_{(0,1) \times \{x\}} {\psi_p^T}^*{F^T}^*\omega^i_j\right) \phi^j_i.\]

Let $\xi^T_p(x) := \eta(dF^T \circ d\psi_p^T(\pdiff{}{s}|_{(x,s)}))$ be the $\mathfrak{so}(n)$-valued function representing counterclockwise rotation around $p$ at the angular speed $r_{p,e}^T(x)$. The fact that $\xi$ is independent of $s$ is a consequence of condition 3. Then since $\omega(X) = \eta(X)$ for any vertical vector $X$, the integral can be made simpler:
\begin{align*}
\int_{\mathring{p}}\left(\int_{(0,1) \times \{x\}}{\psi_p^T}^*{F^T}^*\omega^i_j\right)\phi^j_i 
&= \int_{\mathring{p}}\left(\int_0^1\omega^i_j\left(dF^T \circ d\psi_p^T\left(\pdiff{}{s}\right)\right)ds\right)\phi^j_i\\
&= \int_{\mathring{p}}\left(\int_0^1(\xi^T_p)^i_j(x)ds\right)\phi^j_i = \int_{\mathring{p}}(\xi^T_p)^i_j\phi^j_i.
\end{align*}
This means we can simplify the final term of (\ref{distcurvpart1}) into a single integral over $\mathring{p}$:
\begin{equation}\label{angledefectpart1} 
\sum_{T \supset p} \int_{{\Phi^T}^{-1}(\mathring{p})}{F^T}^* ( \omega^i_j \wedge \pi^*\phi^j_i ) = \int_{\mathring{p}}\iprod{\sum_{T \supset p} \xi^T_p \wedge \phi}.
\end{equation}

The vector $\xi^T_p \in \mathfrak{so}(n)$ can be explicitly calculated, using condition 3 of the compatible frame, as $\mathrm{Ad}((A_{p,e}^T\mu_e^T)^{-1})(r^T_{p,e}w^{n-1}_n)$. The sum of these terms is the derivative of 
\[
G_p(s) := \prod_{j = k}^1\mathbf{Ad}((A_{p,e_j}^{T_j}\mu_{e_j}^{T_j})^{-1})\left(\exp\left(s~r_{p,e_j}^{T_j}w^{n-1}_n\right)\right)
\]
at $s = 0$, where $T_j$ and $e_j$ are some enumeration of the codimension-0 and codimension-1 faces surrounding $p$ such that for each $j=1,\dots,{k-1}$, $T_j\cap T_{j+1} = e_{j+1}$, the rotation from $E_{p,e_j}^{T_j}$ to $E_{p,e_{j+1}}^{T_j}$ is in the counterclockwise direction for each $j = 1,\dots, k-1$, and $k$ is the number of codimension-1 faces incident to $p$. Another way to characterize this enumeration is that $E_{p,e_{j}}^{T_j}$ has an inward-pointing normal vector as its last element and $E_{p,e_{j+1}}^{T_j}$ has an outward-pointing normal vector as its last element. If $p$ is an interior polytope, then we also require that the preceding statements hold for $j=k$, where $T_{k+1}:=T_1$ and $e_{k+1}:=e_1$.

Note that above, the multiplication is arranged so that $j = k$ is the leftmost factor and $j = 1$ is the rightmost factor. The order of multiplication here technically doesn't matter, as the derivative of $G_p(s)$ at $s = 0$ is the same regardless of the ordering of multiplications, but this is the convention we chose.

Next we will use algebraic manipulations much like those in equation (\ref{continuousframejump}) to derive an expression for $A_{p,e_{j+1}}^{T_{j+1}}\mu_{e_{j+1}}^{T_{j+1}}(A_{p,e_j}^{T_j}\mu_{e_j}^{T_j})^{-1}$.

For $x \in p$, condition 3 on the compatible frame (and the fact that $F^{T_j}$ is continuous) implies 
\[
f^{T_j}|_{e_{j+1}}(x) = F^{T_j}(\psi_p^{T_j}(1,x)) = E_{p,e_j}^{T_j} \cdot \exp\left(r_{p,e_j}^{T_j}(x)w^{n-1}_n\right)A_{p,e_j}^{T_j}\mu_{e_j}^{T_j}.
\]

However, we also know that $f^{T_j}|_{e_{j+1}} = E_{p,e_{j+1}}^{T_j} \cdot A_{p,e_{j+1}}^{T_j}\mu_{e_{j+1}}^{T_j}$, so because the rotation from $E_{p,e_j}^{T_j}$ to $E_{p,e_{j+1}}^{T_j}$ is in the counterclockwise direction, we have $E_{p,e_{j+1}}^{T_j} = E_{p,{e_j}}^{T_j} \cdot \exp\left(\theta_p^{T_j}w^{n-1}_n\right)$, so we can use the fact that the group action is free on each fiber of $\mathcal{F}_O(T_j)$ to derive
\[
\exp\left(r_{p,e_j}^{T_j}w^{n-1}_n\right)A_{p,e_j}^{T_j}\mu_{e_j}^{T_j} = \exp\left(\theta_p^{T_j}w^{n-1}_n\right)A_{p,e_{j+1}}^{T_j}\mu_{e_{j+1}}^{T_j}.
\]
Lastly we can apply the fact that $A_{p,e_{j+1}}^{T_j}\mu_{e_{j+1}}^{T_j} = A_{p,e_{j+1}}^{T_{j+1}}\mu_{e_{j+1}}^{T_{j+1}}$ to derive
\begin{equation}\label{discreteframejump}
A_{p,e_{j+1}}^{T_{j+1}}\mu_{e_{j+1}}^{T_{j+1}}(A_{p,e_j}^{T_j}\mu_{e_j}^{T_j})^{-1} = \exp\left((r_{p,e_j}^{T_j} - \theta^{T_j}_p)w^{n-1}_n\right).
\end{equation}
Then $G_p(s) =  (A_{p,e_k}^{T_k}\mu_{e_k}^{T_k})^{-1}\exp\left(\gamma(s)w^{n-1}_n\right)A_{p,e_1}^{T_1}\mu_{e_1}^{T_1}$, where 
\[\gamma(s) := sr_{p,e_k}^{T_k} + \sum_{j=1}^{k-1} (1 + s)r_{p,e_j}^{T_j} - \theta_p^{T_j}.\] 

For interior polytopes, this can be simplified with one more application of (\ref{discreteframejump}) to obtain the equivalent expression
\begin{equation}\label{Gpexpr2}
G_p(s) = \mathbf{Ad}((A_{p,e_k}^{T_k}\mu_{e_k}^{T_k})^{-1})\left(\exp\left[(\gamma(s) + r_{p,e_k}^{T_k} - \theta_p^{T_k})w^{n-1}_n\right]\right).
\end{equation}

Therefore for interior polytopes, $G_p(0) = I$ implies that $\gamma(0) + r_{p,e_k}^{T_k} - \theta_p^{T_k} = 2\pi m$ for an integer $m$. The integer $m$ must be continuous with respect to continuous deformations of $g$ and $f$.  Thus, using the final condition for compatible frames and the fact that $r_{p,e}^T = 0$ for continuous frames and continuous metrics, we get $m = -1$.  Therefore $\sum_{j=1}^k r_{p,e_j}^{T_j} = -\Theta_p$, where $\Theta_p := 2\pi - \sum_{j = 1}^k \theta_p^{T_j}$ is the angle defect of $p$ at $x$. Lastly we note that, due to equation (\ref{Gpexpr2}), 
\[
\left.\pdiff{}{s}\right|_{s = 0}G_p(s) = \mathrm{Ad}\left((A_{p,e_k}^{T_k}\mu_{e_k}^{T_k})^{-1}\right)\left(\sum_{j = 1}^k r_{p,e_j}^{T_j} w^{n-1}_n\right),
\]
so we finally get 
\begin{equation}\label{angledefectpart2}
\int_{\mathring{p}}\iprod{\sum_{T \supset p} \xi^T_p \wedge \phi} = -\int_{\mathring{p}} \Theta_p\iprod{\mathrm{Ad}((A_{p,e_k}^{T_k}\mu_{e_k}^{T_k})^{-1})(w^{n-1}_n) \wedge \phi}.
\end{equation}
The left-hand side of this equation clearly doesn't depend on the particular enumeration $T_j$ we picked, so we will omit the subscript ${}_k$ from now on and remember that although the expression on the right-hand side involves particular codimension-0 and codimension-1 polytopes $T$ and $e$ of which $p$ is a face, the expression is actually independent of the particular choice. The only caveat is that the integral needs to be evaluated with the orientation on $p$ induced from $e$, and $e$ needs to be chosen so that the last entry of $E_{p,e}^T$ is an inward-pointing normal vector; otherwise the sign on $\Theta_p$ would need to flip.

\subsection{Properties of the Distributional Curvature}

Combining~\eqref{distcurvpart1},~\eqref{curvjump},~\eqref{angledefectpart1}, and~\eqref{angledefectpart2}, we arrive at our final expression for the distributional curvature:
\begin{theorem}\label{distcurvthrm} 
Suppose $f = \bigsqcup_{T \subseteq M} f^T$ satisfies all of the compatibility conditions in Definition \ref{compatibilitydef} and $\phi \in C^\infty_c\Omega^{n-2}(M;\mathfrak{so}(n))$ is a smooth compactly supported $\mathfrak{so}(n)$-valued $(n-2)$-form which vanishes when pulled back to $\partial M$. Then 
\begin{align}\label{distcurvfinal}
\iprod{\iprod{f^*\Omega_\mathrm{dist},\phi}} = \sum_{T \subseteq M}\int_{\mathring{T}}&\iprod{{f^T}^*\Omega \wedge \phi} - \sum_{\mathring{e} \subset \mathring{M}}\int_{\mathring{e}}\iprod{\mathrm{Ad}((\mu_e^T)^{-1})([\![\Two_e]\!]) \wedge i_{\mathring{e}}^*\phi} \notag \\
&+ \sum_{\mathring{p} \subset \mathring{M}}\int_{\mathring{p}}\Theta_p\iprod{\mathrm{Ad}((A_{p,e}^T\mu_e^T)^{-1})(w^{n-1}_n) \wedge i_{\mathring{p}}^*\phi}.
\end{align}
Furthermore, the distributional curvature is tensorial, in the following sense: if $h$ is a continuous piecewise-smooth map $M \to O(n)$, then $f \cdot h$ is a compatible frame and 
\[
\iprod{\iprod{(f \cdot h)^*\Omega_{\mathrm{dist}},\phi}} = \iprod{\iprod{\mathrm{Ad}(h^{-1})(f^*\Omega_{\mathrm{dist}}),\phi}}.
\]
\end{theorem}

\begin{proof}
The proof of the expression for the distributional curvature was already carried out in subsections 2.1-2.3, and it is clear that $f \cdot h$ satisfies all the compatibility conditions if $h$ is continuous and piecewise-smooth, so all that remains to be proven is the tensoriality. First, we will expand out the expression for the distributional curvature in the frame $\hat f = f \cdot h$. Note that since $f = E_e^T \cdot \mu_e^T$, $\hat\mu_e^T = \mu_e^Th$, so the expression is not~hard to evaluate:

\begin{align*}
&\iprod{\iprod{\hat{f}^*\Omega_{\mathrm{dist}},\phi}} \\
&= \sum_{T \subseteq M} \int_{\mathring{T}}\iprod{(f^T \cdot h)^*\Omega \wedge \phi} - \sum_{\mathring{e} \subset \mathring{M}}\int_{\mathring{e}}\iprod{\mathrm{Ad}((\mu_e^Th)^{-1})([\![\Two_e]\!]) \wedge i_{\mathring{e}}^*\phi} \\
&\quad + \sum_{\mathring{p} \subset \mathring{M}}\int_{\mathring{p}}\Theta_p\iprod{\mathrm{Ad}((A_{p,e}^T\mu_e^Th)^{-1})(w^{n-1}_n) \wedge i_{\mathring{p}}^*\phi}\\
&= \sum_{T \subseteq M} \int_{\mathring{T}}\iprod{\mathrm{Ad}(h^{-1})({f^T}^*\Omega) \wedge \phi} -\sum_{\mathring{e} \subset \mathring{M}} \int_{\mathring{e}}\iprod{\mathrm{Ad}(h^{-1})(\mathrm{Ad}((\mu_e^T)^{-1})([\![\Two_e]\!])) \wedge i_{\mathring{e}}^*\phi} \\
&\quad + \sum_{\mathring{p} \subset \mathring{M}} \int_{\mathring{p}} \Theta_p\iprod{\mathrm{Ad}(h^{-1})(\mathrm{Ad}((A_{p,e}^T\mu_e^T)^{-1})(w^{n-1}_n)) \wedge i_{\mathring{p}}^*\phi}\\
&= \sum_{T \subseteq M} \int_{\mathring{T}} \iprod{{f^T}^*\Omega \wedge \mathrm{Ad}(h)(\phi))} - \sum_{\mathring{e} \subset \mathring{M}}\int_{\mathring{e}} \iprod{\mathrm{Ad}((\mu_e^T)^{-1})([\![\Two_e]\!]) \wedge \mathrm{Ad}(h)(i_{\mathring{e}}^*\phi)} \\
&\quad + \sum_{\mathring{p} \subset \mathring{M}} \int_{\mathring{p}} \Theta_p\iprod{\mathrm{Ad}((A_{p,e}^T\mu_e^T)^{-1})(w^{n-1}_n) \wedge \mathrm{Ad}(h)(i_{\mathring{p}}^*\phi)}\\
&= \iprod{\iprod{f^*\Omega_{\mathrm{dist}},\mathrm{Ad}(h)(\phi)}}\\
&= \iprod{\iprod{\mathrm{Ad}(h^{-1})(f^*\Omega_{\mathrm{dist}}),\phi}}.
\end{align*}
\end{proof}

One nice consequence of Equation (\ref{distcurvfinal}) is that, if $K \subseteq M$ is compact, then there exists a number $C_K < \infty$ such that for all test forms $\phi$ having support contained in $K$, $|\iprod{\iprod{f^*\Omega_\mathrm{dist},\phi}}| < C_K\sup_{x \in M} \|\phi(x)\|_{g}$. In other words, $f^*\Omega_{\mathrm{dist}}$ is a distribution, or more specifically a current, of order 0. Additionally, since the curvature depends only on integrals of $\phi$ over submanifolds of codimension $\le 2$, a coordinate-dependent $H^2$ norm can be chosen on $K$, and by the Rellich theorem, $\|f^*\Omega_\mathrm{dist}\|_{H^{-2}(K)} < \infty$.

Another observation we can make is that the right-hand side of~\eqref{distcurvfinal} evaluates to zero if $\phi$ is instead $\mathrm{sym}_{n \times n}(\mathbb{R})$-valued, since all of the integrals are against $\mathfrak{so}(n)$-valued forms. This supports the intuitive notion that $f^*\Omega_{\mathrm{dist}}$ is ``$\mathfrak{so}(n)$-valued'', as $\mathfrak{so}(n)$ and $\mathrm{sym}_{n \times n}(\mathbb{R})$ are orthogonal under the bilinear product we are using.

The space of test forms can also be expanded to include discontinuous forms of a particular type. 

\begin{definition}\label{compattestformdef}Consider a compactly supported $\mathfrak{so}(n)$-valued $(n-2)$-form $\phi$ which is $C^2$ on $\mathring{T}$ for each codimension-0 polytope $T \subseteq M$. Such a form will be called a \emph{compatible $\mathfrak{so}(n)$-valued test form} if the following conditions are satisfied:

\begin{enumerate}
\item{For each $T \subseteq M$ there exists a unique Lipschitz continuous extension of ${\Phi^T}^*\phi|_{\mathring{T}}$ to $B_T$, which will be called $\tilde{\phi}^T \in C^0\Omega^{n-2}(B_T;\mathfrak{so}(n))$. In other words, ${\Phi^T}^*\phi|_{\mathring{T}}$ has uniformly bounded first derivatives in any coordinate chart. This implies $i_{\mathring{e}}^*\phi|_T := {\Phi^T}^{-*}i_{{\Phi^T}^{-1}(\mathring{e})}^*\tilde{\phi}^T$ is a well-defined continuous extension of $\phi|_{\mathring{T}}$. It also implies that if $d\Phi^{T}(v) = 0$, then $v \lrcorner\tilde{\phi}^T = 0$ (this is relevant only on the boundary of $B_T$).}

\item{If $e = T \cap T'$ is a codimension-1 interface, then $i_{\mathring{e}}^*\phi|_T = i_{\mathring{e}}^*\phi|_{T'}$. In other words, $[\![i_{\mathring{e}}^*\phi]\!] = 0$.}
\item{The form ${\psi_p^T}^*\tilde\phi^T$ is basic for the fibration $q: [0,1] \times p \to p$. More precisely, if $\chi_{s,p}^T(x) :=  \psi_p^T(s,x)$, Then ${\chi_{s,p}^T}^*\tilde{\phi}^T = {\chi_{0,p}^T}^*\tilde{\phi}^T$ for any $s \in [0,1]$. In other words, $i_{\mathring{p}}^*\phi|_T$ is well-defined and continuous on $p$, although $\phi$ itself may be discontinuous at $p$. }
\item{If $\Delta_d \subset \partial M$, then $i_{{\Phi^{T}}^{-1}(\mathring{\Delta}_d)}^*\tilde{\phi}^T = 0$ for each $T \supset \Delta_d$. In other words, $i_{\partial M}^*\phi = 0$.}
\end{enumerate}
The set of compatible test forms for the manifold $M$ with mesh $\mathcal{T}$ will be called $\mathcal{C}(\mathcal{T},M)$.
\end{definition}

If $\phi$ is a compatible $\mathfrak{so}(n)$-valued test form, then the proofs in Sections \ref{integrationbypartssection}-\ref{angledefectsection} remain valid with the small modification of using $\tilde\phi^T$ in place of ${\Phi^T}^*\phi$. Note that items (2) and (3) together imply that $i_{\mathring{p}}^*\phi = {\chi_{s,p}^T}^*\tilde\phi^T$ is well-defined independently of $s$ and $T$.  A notable example of compatible test forms are the forms $\phi_{ij} \otimes \star f^*(\theta^i \wedge \theta^j)$, where each $\phi_{ij}: M \to \mathfrak{so}(n)$, $i,j=1,2,\dots,n$, is a compactly supported continuous map that is $C^2$ on the interior of each codimension-0 polytope and vanishes on $\partial M$. The map $\star: \Omega^2(M) \to \Omega^{n-2}(M)$ is, on each codimension-0 polytope $T$, the Hodge star operator associated to the metric $g^T$.

\subsection{Removing Frame-Dependency}\label{deframesection}

While useful for analysis, (\ref{distcurvfinal}) is inconvenient for many applications because it requires knowledge of a specific compatible frame, which are not easy to construct or represent. The purpose of this subsection is to remove this dependency by expressing the distributional curvature in terms of endomorphism-valued forms. The cost is that the test forms must be discontinuous and metric-dependent.

First, we can define a symmetric nondegenerate bilinear form on the vector space $\mathrm{End}(T_xM) = T_xM \otimes T^*_xM$ by setting 

\begin{equation}\label{endoprod}\iprod{\alpha \otimes v,\beta \otimes w} := \alpha(w)\beta(v)\end{equation}
for $\alpha, \beta \in T_x^*(M)$, $v,w \in T_xM$ and extending multilinearly. Expressed in a basis $\{F_i\}_{i = 1}^n$ and its dual basis $\{F^i\}_{i = 1}^n$, this would be $\iprod{A^i_j F_i\otimes F^j,B^k_lF_k \otimes F^l} = A^i_jB^j_i$. Therefore, picking any basis for $T_xM$ gives a linear isometry between $\mathfrak{gl}(n)$ and $\mathrm{End}(T_xM)$, by setting $F^j(\hat{\phi}^T(F_i)) := \phi^j_i$. The left-hand side of this equation could be more succinctly written $([\hat{\phi}]_{F})^j_i$, where $[\hat{\phi}]_F$ is the matrix representation of the linear map $\hat{\phi}: T_xM \to T_xM$ in the basis $F$. In particular, for $y \in B_T$, the orthonormal basis $F^T(y)$ gives a special isometry $\Psi_y^T: \mathfrak{gl}(n) \to \mathrm{End}(T_{\Phi^T(y)}M)$. 

This pointwise isometry in turn defines a map $\Psi_f: \bigsqcup_{T \subseteq M}C^0\Omega^{n-2}(\mathring{T};\mathfrak{gl}(n)) \to \bigsqcup_{T \subseteq M}C^0\Omega^{n-2}(\mathring{T};\mathrm{End}(TM))$, by setting $\Psi_f(\phi)|_x$ equal to $\Psi_{{\Phi^T}^{-1}(x)}^T(\phi|_x)$ if $x$ is in $\mathring{T}$. Clearly, because $\Phi_y^T$ is invertible for all $y \in \mathring{B}_T$ and $\Phi^T$ is a diffeomorphism when restricted to $\mathring{B}_T$, $\Psi_f$ is a bijection, with inverse given by $\Psi_f^{-1}(\hat{\phi})|_x := (\Psi_{{\Phi^T}^{-1}(x)}^T)^{-1}(\hat{\phi}|_x)$.

\begin{definition}
Let $f$ be a compatible frame. A \emph{compatible $\mathrm{End}(TM)$-valued test form} is the image of a compatible $\mathfrak{so}(n)$-valued test form (see Definition \ref{compattestformdef}) under the map $\Psi_f$. The space of compatible $\mathrm{End}(TM)$-valued test forms is denoted $\mathcal{A}(f,\mathcal{T},M) := \Psi_f(\mathcal{C}(\mathcal{T},M))$.
\end{definition}

Several properties of $\mathrm{End}(TM)$-valued test forms make them more suited to frame-independent computations.

Let $y \in B_T$ and $h \in O(n)$, and let $E$ be the basis of $T_{\Phi^T(y)}M$ determined by $F = E \cdot h$. Then the adjoint action of $h$ changes $\Psi_y^T(u)$ by essentially changing the basis it is represented in:

\begin{align}
F^j\left(\Psi_y^T(\mathrm{Ad}(h)(u))(F_i)\right) 
&= h^j_lu^l_k(h^{-1})^k_i \notag \\
&= h^j_l F^l\left(\Psi_y^T(u)(F_k)\right) (h^{-1})^k_i\notag \\ 
&= h^j_l F^l\left(\Psi_y^T(u)(F_k(h^{-1})^k_i)\right) \notag \\ 
&= E^j\left(\Psi_y^T(u)(E_i)\right). \label{endochangeofbasis}
\end{align}

Since $F^T$ is continuous and Lipschitz on $B_T$ and ${\Phi^T}^*\phi|_{\mathring{T}}$ has a unique Lipschitz extension $\tilde{\phi}^T$ on $B_T$, there exists a unique Lipschitz extension of ${\Phi^T}^*\hat{\phi}|_{\mathring{T}}$ to $B_T$, which at each point $y \in B_T$ is equal to $\Psi_y^T(\tilde{\phi}^T|_y)$. We will use $\hat{\phi}^T$ to refer to this extension of ${\Phi^T}^*\hat{\phi}|_{\mathring{T}}$. Similarly to the case for compatible $\mathfrak{so}(n)$-valued test forms, the pullback $i_{\mathring{e}}^*\hat\phi|_{T} := {\Phi^T}^{-*}i_{{\Phi^T}^{-1}(\mathring{e})}^*\hat\phi^T$ is a well defined continuous extension of $\hat\phi|_{\mathring{T}}$.

The two key properties we will need about compatible $\mathrm{End}(TM)$-valued test forms are proved in the following lemma.

\begin{lemma}
Let $\hat\phi = \Psi_f(\phi)$ be a compatible $\mathrm{End}(TM)$-valued test form. Then the following change of basis equation holds for each $T,T' \subseteq M$ and $e = T \cap T'$:

\begin{equation}\label{endtestform1}
\![i_{\mathring{e}}^*\hat{\phi}|_{T}]_{E_e^{T}} = \mathrm{Ad}\left(\begin{bmatrix}I & 0 \\ 0 & -1\end{bmatrix}\right)\left([i_{\mathring{e}}^*\hat{\phi}|_{T'}]_{E_e^{T'}}\right).
\end{equation}

Additionally, the following equation holds for each $p \subset e \subset T$ such that $E_{p,e}^T$ has an inward-pointing normal vector as its lasty entry:

\begin{equation}\label{endtestform2}
[{\chi_{s,p}^T}^*\hat{\phi}^T]_{E_{p,e}^T} = \mathrm{Ad}\left(\exp\left(s~r_{p,e}^Tw^{n-1}_n\right)\right) \left([{\chi_{0,p}^T}^*\hat{\phi}^T]_{E_{p,e}^T}\right).
\end{equation}
Here, $\chi_{s,p}^T$ is the map defined within condition 3 of Definition~\ref{compattestformdef}.
\end{lemma}

\begin{proof}
By (\ref{endochangeofbasis}), at points $x \in e = T \cap T'$, we have on $T$ that $\mathrm{Ad}((\mu_e^T)^{-1})([i_{\mathring{e}}^*\hat{\phi}|_T]_{E_e^T}) = i_{\mathring{e}}^*\phi = \mathrm{Ad}((\mu_e^{T'})^{-1})([i_{\mathring{e}}^*\hat{\phi}|_{T'}]_{E^{T'}_e})$, since $\phi$ is compatible. However, because $\mu_e^{T'} = \begin{bmatrix}I & 0 \\ 0 & -1\end{bmatrix}\mu_e^T$, we have
\[
\mathrm{Ad}((\mu_e^T)^{-1})\left([i_{\mathring{e}}^*\hat\phi|_T]_{E_e^T} - \mathrm{Ad}\left(\begin{bmatrix}I & 0 \\ 0 & -1\end{bmatrix}\right)\left([i_{\mathring{e}}^*\hat\phi|_{T'}]_{E^{T'}_e}\right)\right) = 0
\]
By applying $\mathrm{Ad}(\mu_e^T)$ to both sides of this equation, we get Equation (\ref{endtestform1}). 

Similarly, the form $[i_{\mathring{p}}^*\hat{\phi}^T]_{F^T} = i_{\mathring{p}}^*\phi$ is well-defined, in the sense that ${\chi_{s,p}^T}^*\tilde{\phi}^T$ does not depend on $s$ or $T$. Below we will use the shorthand $F(s,x) = F^T\circ \psi_p^T(s,x)$. 

By \eqref{endochangeofbasis} and condition 3 of compatible frames, for any $s \in [0,1]$, we have
\[ 
[{\chi_{s,p}^T}^*\hat{\phi}^T]_{E_{p,e}^T} = \mathrm{Ad}\left(\exp\left(s~r_{p,e}^Tw^{n-1}_n\right)A_{p,e}^T\mu_e^T\right)\left([{\chi_{s,p}^T}^*\hat{\phi}^T]_{F(s,\cdot)}\right)
\]
and 
\[[{\chi_{0,p}^T}^*\hat{\phi}^T]_{E_{p,e}^T} = \mathrm{Ad}(A_{p,e}^T\mu_e^T)([{\chi_{0,p}^T}^*\hat{\phi}^T]_{F(0,\cdot)}).\]
Since ${\chi_{s,p}^T}^*\tilde{\phi}^T = {\chi_{0,p}^T}^*\tilde{\phi}^T$ is one of the defining characteristics of a compatible $\mathfrak{so}(n)$-valued test form and $[{\chi_{s,p}^T}^*\hat{\phi}^T]_{F(s,\cdot)} = {\chi_{s,p}^T}^*\tilde{\phi}^T$, the second equation can be substituted into the first to obtain Equation (\ref{endtestform2}).

\end{proof}

A more intuitive way to understand what this lemma gives us is: rather than $\hat{\phi}$ itself being continuous, we get that the matrix representations of $\hat{\phi}$ in special frames adapted to the mesh are only allowed to have discontinuties in a special way, and except for the dependence on $r_{p,e}^T$, the discontinuities do not depend on the frame $f$.

Just as for compatible test forms, an important class of compatible $\mathrm{End}(TM)$-valued test forms is given by the forms $\phi^k_{lij} (f_k \otimes f^*\theta^l) \otimes \star f^*(\theta^i \wedge \theta^j)$, where the coefficients $\phi^k_{lij}$ are continuous compactly supported functions on $M$ which are $C^2$ on the interior of each codimension-0 polytope and vanish on $\partial M$ and which alternate in the $k,l$ indices. These are exactly the images of the forms $\phi_{ij}' \otimes \star f^*(\theta^i \wedge \theta^j)$ under the map $\Psi_f$, where $\phi_{ij}': M \to \mathfrak{so}(n)$ has the same continuity and smoothness as the maps $\phi^k_{lij}$. Unlike the case for compatible $\mathfrak{so}(n)$-valued test forms, the set of compatible $\mathrm{End}(TM)$-valued test forms is in general not a superset of the smooth $\mathrm{End}(TM)$-valued test forms, as discontinuities in $f$ can force discontinuties in $\hat\phi$.

Note that if $\hat \phi \in \mathcal{A}(f,\mathcal{T},M)$, then $\hat\phi \in \mathcal{A}(f \cdot h,\mathcal{T},M)$ for any continuous piecewise-$C^2$ map $h: M \to O(n)$, so the dependence of $\mathcal{A}(f,\mathcal{T},M)$ on the choice of frame is limited to dependence on its discontinuities. Additionally, the validity of equations (\ref{endtestform1}) and (\ref{endtestform2}) does not depend on the choice of tangent frames $E_e$ and $E_p$, because if $\tilde E_e$ is another choice of orthonormal tangent frame for $e$, then $\tilde E_e^T = E_e^T \cdot \begin{bmatrix}h & 0 \\ 0 & 1\end{bmatrix}$ for some orthogonal matrix $h \in O(n-1)$. Therefore, we have on $\Phi^{-1}(\mathring{e})$, 
\begin{align*}
[i_{\mathring{e}}^*\hat{\phi}^T]_{\tilde E_e^T} &= \mathrm{Ad}\left(\begin{bmatrix}h^{-1} & 0 \\ 0 & 1\end{bmatrix}\right)\left([i_{\mathring{e}}^*\hat\phi^T]_{E_e^T}\right) = \mathrm{Ad}\left(\begin{bmatrix}h^{-1} & 0 \\ 0 & 1\end{bmatrix}\begin{bmatrix}I & 0 \\ 0 & -1\end{bmatrix}\right)\left([i_{\mathring{e}}^*\hat\phi^{T'}]_{E_e^{T'}}\right)\\
&= \mathrm{Ad}\left(\begin{bmatrix}h^{-1} & 0 \\ 0 & 1\end{bmatrix}\begin{bmatrix}I & 0 \\ 0 & -1\end{bmatrix}\begin{bmatrix}h & 0 \\ 0 & 1\end{bmatrix}\right)\left([i_{\mathring{e}}^*\hat \phi^{T'}]_{\tilde E_e^{T'}}\right)\\
&= \mathrm{Ad}\left(\begin{bmatrix}I & 0 \\ 0 & -1\end{bmatrix}\right)\left([i_{\mathring{e}}^*\hat\phi^{T'}]_{\tilde E_e^{T'}}\right).
\end{align*}

Similar reasoning shows that if an alternative (similarly oriented) tangential orthonormal frame $\tilde E_p$ is chosen for $p$, then equation (\ref{endtestform2}) is valid for $[{\chi_{s,p}^T}^*\hat\phi^T]_{\tilde E_{p,e}^T}$ as well.

We may also define bijections $\Psi_{E_e^T}: C^0\Omega^k(e;\mathfrak{gl}(n)) \to C^0\Omega^k(e;\mathrm{End}(TM))$ and $\Psi_{E_{p,e}^T}: C^0\Omega^k(p;\mathfrak{gl}(n)) \to C^0\Omega^k(p;\mathrm{End}(TM))$ in exactly the same way as $\Psi_f$, but using the frames $E_e^T$ and $E_{p,e}^T$ defined on $e$ and $p$ respectively.

We can then define the following endomorphism-valued forms, which are a $2$-form defined on the codimension-0 polytope $T$, a pair of $1$-forms defined on the codimension-1 polytope $e = T \cap T'$, and a $0$-form defined on the codimension-2 polytope $p \subset T \cap e$ respectively:
\begin{align*}
\hat{R} &:= \Psi_f(f^*\Omega),\\
\hat{\Two}_e^T &:= \Psi_{E_e^T}(\Two_e^T),\\
\hat{\Two}_e^{T'} &:= \Psi_{E_e^{T'}}(\Two_e^{T'}),\\
\hat{\Theta}_p^T &:= \Psi_{E_{p,e}^T}(\Theta_pw^{n-1}_n) = \Theta_p(-\vec{n} \otimes \vec{\nu}^\flat + \vec{\nu}\otimes\vec{n}^\flat) .
\end{align*}

Note that for consistency, when defining $\hat\Theta_p^T$, we must take care to use the face $e$ such that $E_{p,e}^T$ has $-\vec{n}^T$ as its last vector, so that the rotation from $E_{p,e}^T$ to $E_{p,e'}^T$ ($e'$ being the other codimension-1 face meeting $p$) is in the counterclockwise direction. Otherwise, as mentioned at the end of section \ref{angledefectsection}, the sign of $\Theta_p$ will need to flip.

In coordinates, these forms can explicitly be written as
\begin{align}
\iprod{\hat{R}f_j,f_i}_T &= {f^T}^*\Omega^i_j, \nonumber \\
\iprod{\hat{\Two_e}^T( E_e^T)_j,(E_e^T)_i}_T &= \begin{cases}0 &\text{if } i,j < n ,\\
\iprod{\nabla^T((E_e^T)_j),(E_e^T)_i}_T &\text{if } i = n\;\text{or}\; j = n,\end{cases} \label{TwoEndCoord}\\
\iprod{\hat{\Theta}_p^T(E_{p,e}^T)_j,(E_{p,e}^T)_i}_T &= \begin{cases}\Theta_p &\text{if } i = n, j = n-1,\\
-\Theta_p &\text{if } i = n-1,j = n, \nonumber\\
0 &\text{otherwise.}\end{cases} 
\end{align}

Here $\iprod{\cdot,\cdot}_T$ and $\nabla^T$ refer to the inner product and covariant derivative induced by the metric $g^T$, respectively.

Frame-independent expressions for $\hat{R}$, $\hat{\Two}_e^T$, and $\hat{\Theta}_p^T$ are given in the lemma below.

\begin{lemma}
On each $T \subseteq M$, $\hat{R}$ is the usual Riemann curvature tensor, i.e.~it is the $\operatorname{End}(TM)$-valued 2-form given by
\[
\hat{R}(X,Y)Z = \nabla_X\nabla_YZ - \nabla_Y\nabla_XZ - \nabla_{[X,Y]}Z.
\]
On each $e \subset T$, $\hat{\Two}_e^T$ is given as follows.  For all $x \in \mathring{e}$, $X,Y \in T_x M$, and $Z \in T_x e$, 
\[
\langle \hat{\Two}_e^T(Z)X, Y \rangle_T = \langle \nabla^T_Z \vec{n},Y \rangle_T \langle X, \vec{n} \rangle_T - \langle \nabla^T_Z \vec{n},X \rangle_T \langle Y, \vec{n} \rangle_T.
\]
On each $p \subset T$, $\hat{\Theta}_p^T$ is the $\operatorname{End}(TM)$-valued 0-form given by multiplying the angle defect $\Theta_p$ times the infinitesimal generator of counterclockwise rotation in the plane orthogonal to $p$ in $T$.  
\end{lemma}

\begin{proof}
The claims about $\hat{R}$ and $\hat{\Theta}_p^T$ are clear from their definitions.  To derive the formula for $\hat{\Two}_e^T$, let $E_i$ be shorthand for $(E_e^T)_i$, and let $X=X^i E_i$, $Y=Y^i E_i$, and $Z=Z^i E_i$.  Assume that $Z$ is tangent to $e$, so that $Z^n=0$.  Then
\begin{align*}
&\langle \nabla^T_Z \vec{n},Y \rangle_T \langle X, \vec{n} \rangle_T - \langle \nabla^T_Z \vec{n},X \rangle_T \langle Y, \vec{n} \rangle_T \\
&= Y^i \langle \nabla_Z^T E_n, E_i \rangle_T X^n - X^j \langle \nabla_Z^T E_n, E_j \rangle_T Y^n \\
&= X^n \langle \nabla_Z^T E_n, E_i \rangle_T Y^i + X^j \langle \nabla_Z^T E_j, E_n \rangle_T Y^n \\
&= \sum_{\substack{i,j=1 \\ i=n \text{ or } j=n}}^n X^j \langle \nabla_Z^T E_j, E_i \rangle_T Y^i.
\end{align*}
Above, we used the fact that $\langle \nabla_Z^T E_n, E_j \rangle_T = -\langle \nabla_Z^T E_j, E_n \rangle_T$.  We get the claimed result upon comparison with~\eqref{TwoEndCoord}.
\end{proof}

If $f$ is any compatible frame and $\hat \phi \in \mathcal{A}(f,\mathcal{T},M)$, then we define
\begin{align*}
\iprod{\iprod{\hat{R}_\mathrm{dist},\hat{\phi}}} &:= \iprod{\iprod{f^*\Omega_\mathrm{dist},\Psi_f^{-1}(\hat\phi)}}.
\end{align*}

The following theorem gives a more explicit formula for $\hat{R}_{\mathrm{dist}}$ and demonstrates that $\hat{R}_{\mathrm{dist}}$ is well-defined.

\begin{theorem} \label{Rdistthrm}
If $\hat \phi \in \mathcal{A}(f,\mathcal{T},M)$ for some compatible frame $f$, then
\begin{align}
\iprod{\iprod{\hat{R}_\mathrm{dist},\hat{\phi}}} =
\sum_{T \subseteq M} \int_{\mathring{T}}\iprod{\hat{R}\wedge \hat{\phi}} - \sum_{\mathring{e} \subset \mathring{M}} \int_{\mathring{e}} \left[\!\left[\iprod{\hat{\Two}_e \wedge i_{\mathring{e}}^*\hat{\phi}}\right]\!\right] + \sum_{\mathring{p} \subset \mathring{M}}\int_{\mathring{p}}\iprod{\hat{\Theta}_p\wedge i_{\mathring{p}}^*\hat{\phi}}.  \label{distcurvendofinal}
\end{align}
Here $\left[\!\left[\iprod{\hat{\Two}_e \wedge i_{\mathring{e}}^*\hat\phi}\right]\!\right]$ is defined as $\iprod{\hat{\Two}_e^T \wedge i_{\mathring{e}}^*\hat\phi|_T} - \iprod{\hat\Two_e^{T'} \wedge i_{\mathring{e}}^*\hat{\phi}|_{T'}}$, which is interpreted as the jump in second fundamental form, while $\iprod{\hat\Theta_p \wedge i_{\mathring{p}}^*\hat\phi}$ is defined as $\iprod{\hat\Theta_p^T \wedge i_{\mathring{p}}^*\hat\phi|_T}$, where $T$ is any of the $n$-dimensional polytopes containing $p$ as a face.
\end{theorem}

\begin{proof}
This follows fairly simply from the definitions (\ref{TwoEndCoord}) and the change of basis formula (\ref{endochangeofbasis}). Let $\phi = \Psi_f^{-1}(\hat\phi)$, which is a compatible test form. Then:
\begin{align*}
\iprod{\hat{R} \wedge \hat\phi} &= \iprod{\Psi_f^{-1}(\hat{R})\wedge \Psi_f^{-1}(\hat{\phi})} = \iprod{{f^T}^*\Omega \wedge \phi}\\
\iprod{\hat\Two_e^T \wedge i_{\mathring{e}}^*\hat\phi|_T} &= \iprod{\Psi_{E_e^T}^{-1}(\hat\Two_e^T) \wedge \Psi_{E_e^T}^{-1}(i_{\mathring{e}}^*\hat\phi|_T)} = \iprod{\Two_e^T \wedge \mathrm{Ad}(\mu_e^T)(i_{\mathring{e}}^*\phi)}\\
\iprod{\hat\Theta_p^T \wedge i_{\mathring{p}}^*\hat\phi|_T} &= \iprod{\Psi_{E_{p,e}^T}^{-1}(\hat\Theta_p^T) \wedge \Psi_{E_{p,e}^T}^{-1}(i_{\mathring{p}}^*\hat\phi|_T)} = \iprod{\Theta_pw^{n-1}_n \wedge \mathrm{Ad}(A_{p,e}^T\mu_e^T)(i_{\mathring{p}}^*\phi)}.
\end{align*}

From here, we can apply the identity $\iprod{\alpha \wedge \mathrm{Ad}(h)(\beta)} = \iprod{\mathrm{Ad}(h^{-1})(\alpha) \wedge \beta}$ which is valid for $\mathfrak{so}(n)$-valued forms. The only complication arises from the jump in second fundamental form. Note that because $\mu_e^{T'} = \begin{bmatrix}I & 0 \\0 & -1\end{bmatrix}\mu_e^T$, we have
\begin{align*}\iprod{\mathrm{Ad}((\mu_e^T)^{-1})(\Two_e^T) \wedge i_{\mathring{e}}^*\phi} -& \iprod{\mathrm{Ad}((\mu_e^{T'})^{-1})(\Two_e^{T'})\wedge i_{\mathring{e}}^*\phi} \\
&= \iprod{\mathrm{Ad}((\mu_e^T)^{-1})(\Two_e^T) - \mathrm{Ad}((\mu_e^T)^{-1})(-\Two_e^{T'}) \wedge i_{\mathring{e}}^*\phi}\\
&= \iprod{\mathrm{Ad}((\mu_e^T)^{-1})([\![\Two_e]\!]) \wedge i_{\mathring{e}}^*\phi}
\end{align*}
\end{proof}

Note that, since $i_{\mathring{p}}^*\hat{\phi}|_T$ is not single-valued due to \eqref{endtestform2}, at first glance the last term seems to have some ambiguity. However, $\hat{\Theta}_p^T$ is invariant under rotations of the normal vectors $\vec{\nu}$ and $\vec{n}$ with which it is defined. More precisely, if $h = \exp\left(s~r_{p,e}^T(x)w^{n-1}_n\right)$, then $[\hat{\Theta}_p^T]_{E_{p,e}^T} = [\hat{\Theta}_p^T]_{E_{p,e}^T \cdot h^{-1}}$. So, because $[{\chi_{s,p}^T}^*\hat{\phi}^T]_{E_{p,e}^T} = \mathrm{Ad}(h)([{\chi_{0,p}^T}^*\hat{\phi}^T]_{E_{p,e}^T})$ by \eqref{endtestform2} and $\Psi_{E_{p,e}^T}$ is an isometry, we have
\begin{align*}
\iprod{\hat\Theta_p^T \wedge{\chi_{s,p}^T}^*\hat{\phi}^T}& = \iprod{[\hat{\Theta}_p^T]_{E_{p,e}^T} \wedge [{\chi_{s,p}^T}^*\hat{\phi}^T]_{E_{p,e}^T}}\\
&= \iprod{[\hat{\Theta}_p^T]_{E_{p,e}^T} \wedge \mathrm{Ad}(h)([{\chi_{0,p}^T}^*\hat{\phi}^T]_{E_{p,e}^T})} \\
&= \iprod{\mathrm{Ad}(h^{-1})([\hat{\Theta}_p^T]_{E_{p,e}^T}) \wedge [{\chi_{0,p}^T}^*\hat{\phi}^T]_{E_{p,e}^T}}\\
&= \iprod{[\hat{\Theta}_p^T]_{E_{p,e}^T \cdot h^{-1}} \wedge [{\chi_{0,p}^T}^*\hat{\phi}^T]_{E_{p,e}^T}}\\
&= \iprod{[\hat{\Theta}_p^T]_{E_{p,e}^T} \wedge [{\chi_{0,p}^T}^*\hat{\phi}^T]_{E_{p,e}^T}} \\
&= \iprod{\hat{\Theta}_p^T \wedge {\chi_{0,p}^T}^*\hat{\phi}^T}.
\end{align*}
So, in~\eqref{distcurvendofinal}, ``$i_{\mathring{p}}^*\hat{\phi}|_T$" could be thought of as a shorthand for ${\chi_{s,p}^T}^*\hat{\phi}^T$, where $s$ may be any convenient number in $[0,1]$. Concretely, this means each $x \in \mathring{p}$ can be approached along a ray of any convenient angle, and the result will be the same. Additionally, due to equation (\ref{angledefectpart2}), this term does not depend on the choice of $T$ of which $p$ is a face.

The right side of equation \eqref{distcurvendofinal} does not depend on $f$ at all, so as long as $\hat{\phi} \in \mathcal{A}(f,\mathcal{T},M)$ for some compatible frame $f$, this definition can be used to compute the distributional curvature. The caveats on the integrals are the same as for \eqref{distcurvfinal}: the integrals over $\mathring{e}$ do not depend on the choice of $T$, but care must be chosen so that the integral is evaluated with the orientation induced on $e$ by the orientation of $T$, and the integrals over $\mathring{p}$ must be evaluated using the orientation induced from the face $e$ such that $E_{p,e}^T$ has an inward-pointing normal vector as its last entry.

This expression has some close similarities with the \emph{densitized distributional curvature} investigated in \cite{gopalakrishnan2023analysis}. In fact, the two expressions are equivalent, with the main difference being the choice of how to represent the second fundamental form and which indices are raised/lowered. See Appendix \ref{equivalenceappendix} for a proof. As explained in~\cite{gopalakrishnan2023analysis}, various traces of this distribution can be taken to obtain the Ricci curvature, Einstein tensor, and the scalar curvature.

\section{Construction of Compatible Frames}\label{constructingcompatibleframessection}

Most of Section \ref{distcurvsection} would be meaningless if a compatible frame did not exist. In this section, we will give a proof of the following theorem:

\begin{theorem}\label{compatframeexistencethrm} Let $g$ be a Regge metric for the mesh $\mathcal{T}$ such that each polytope $T \subseteq M$ has a blow-up which is a closed convex polytope in $\mathbb{R}^n$, and suppose there exists a $C^2$ homotopy of Regge metrics $g(t)$ such that $g(1) = g$, $g(0) =: g_0$ is a smooth metric, and there exists a smooth $g_0$-orthonormal frame $f_0$. Then there exists a $C^2$ homotopy of frames $f(t)$ such that $f(0) = f_0$, $f^T(t)$ is $g(t)$-orthonormal when restricted to each $T$, $f(t)$ satisfies conditions 1, 2, and 3 of compatible frames, and the maps $(t,x) \mapsto F^T(t)(x)$ vary continuously as maps $[0,1] \times B_T \to \mathcal{F}_{GL}(T)$, where $F^T$ is the blown-up frame from condition 1.
\end{theorem}

The general proof strategy is as follows: For each codimension-2 polytope $p$, we evolve the frame $E_p$ into a $g(t)$-orthonormal frame $E_p(t)$.  Then, for each codimension-1 polytope $e \supset p$, we produce a frame $E_e(t)$ whose associated matrix $A_{p,e}$ is constant in time. Then, on each codimension-0 polytope $T$, we use Lemma \ref{frameextthrm} (appearing below) to produce homotopies of orthonormal frames $F^T(t): B_T \to \mathcal{F}_O(T)$ having the same matrices $\mu_e^T$ as $f_0$ does, and $r_{p,e}^T(t) = \theta^T_p(t) - \theta^T_p(0)$. The frames $f^T(t) = F^T(t) \circ {\Phi^T}^{-1}$ are then compatible by construction. 

First, we need a result that allows us to extend frames. This is a technical result that ultimately relies on the fact that functions can be extended smoothly from closed sets. We will specifically use the theorem as stated in \cite{Whitney2} (see also~\cite[Lemma 2.26]{lee2013smooth}), which can be simplified for our purposes in the way detailed below.  In the remainder of this section, we will stop abusing terminology and use the word ``polytope'' to refer to a genuine polytope (as opposed to the image of a polytope under a smooth embedding).  Our reason for doing so is that the extension procedure that we will soon describe takes place on the blow-up $B_T$, which we will assume to be a polytope.  (In fact we will assume it to be convex.)  We will use the notation $\hat{T}$ for polytopes below; ultimately we will take $\hat{T}=B_T$ when we begin constructing a compatible frame.

\begin{theorem}[\cite{Whitney2}] \label{smoothextthrm}
Let $\hat{T} \subset \mathbb{R}^n$ be a closed, convex polytope. If $a: \hat{T} \to \mathbb{R}$ is a $C^r$ function, $0 \le r \le \infty$, then there exists a $C^r$ extension $A: \mathbb{R}^n \to \mathbb{R}$ such that $A|_{\hat{T}} = a$.
\end{theorem}

We will use this as the base of a lemma allowing us to extend partially defined functions from the boundary of a polytope to the interior.

\begin{lemma} \label{smoothextthrm2}
Let $\hat{T} \subset \mathbb{R}^n$ be a closed, convex $n$-dimensional polytope. Denote by $\{\hat{e}_i\}_{i = 1}^N$ a set of codimension-1 faces of $\hat{T}$, and let $\hat{p}_{ij} := \hat{e}_i \cap \hat{e}_j$. Then for any collection of $C^r$ ($r \ge 1$) functions  $a^i: \hat{e}_i \to \mathbb{R}$ such that $a^i|_{\hat{p}_{ij}} = a^j|_{\hat{p}_{ij}}$ for all $i,j$, there exists a Lipschitz continuous function $A: \hat{T} \to \mathbb{R}$ such that $A|_{\hat{e}_i} = a^i$ for each $i$ and $A|_{\hat{T} \backslash \bigcup_{i,j} \hat{p}_{ij}}$ is a $C^r$ function.
\end{lemma}

Note that if $\hat{T}$ happens to be a manifold with corners, the Whitney extension theorem can be used directly to produce an extension $A$ which is globally $C^r$.

\begin{proof}
    Per Theorem \ref{smoothextthrm}, each $a^i$ can be extended to a $C^r$ function $A^i: \mathbb{R}^n \to \mathbb{R}$. Since $\hat{e}_i$ is a convex polytope of dimension $n-1$, it lies completely in a hyperplane $E_i \subset \mathbb{R}^n$. Set $\lambda^i(x) := \pm\mathrm{dist}(x,E_i)$, where the sign is positive if $x$ is on the same side of $E_i$ as $\hat{T}$ and negative otherwise. This makes it an affine function.
    
    Now let $\hat{\Lambda}_i(x) = \prod_{j \neq i} \lambda^j(x)$ and $\hat{\Lambda}_{ij}(x) = \prod_{k \notin \{i,j\}} \lambda^k(x)$.  The extension we seek is
    \[
    A(x) = \frac{\sum_i A^i(x) \hat{\Lambda}_i(x)}{\sum_i \hat{\Lambda}_i(x)}.
    \]
    This function satisfies $A|_{\hat{e}_i} = a^i$, and $A$ is $C^r$ on the set such that all of the $\lambda^i$'s are non-negative and no more than one of the $\lambda^i$'s is equal to zero, which is exactly the set $\hat{T}\backslash \bigcup_{i,j} \hat{p}_{ij}$.

    All that remains is to prove that the first partial derivatives of $A$ are bounded if $a^i|_{\hat{p}_{ij}} = a^j|_{\hat{p}_{ij}}$ for all $i,j$.  To do this, we compute (abandoning the Einstein summation convention)
    \begin{align*}
    dA &= \frac{\sum_i A^i \sum_{j \neq i} \hat{\Lambda}_{ij} d\lambda^j \left(\sum_i \hat{\Lambda}_i\right) - \left( \sum_i A^i \hat{\Lambda}_i \right)\left( \sum_i \sum_{j \neq i} \hat{\Lambda}_{ij} d\lambda^j \right)}{\left(\sum_i \hat{\Lambda}_i\right)^2} + \frac{\sum_i \hat{\Lambda}_i dA^i}{\sum_i \hat{\Lambda}_i}\\
    &= \frac{\sum_{i,j} \sum_{l \ne i}(A^i - A^j)\hat{\Lambda}_j\hat{\Lambda}_{il}d\lambda^l}{\left(\sum_i \hat{\Lambda}_i\right)^2} + \frac{\sum_i \hat{\Lambda}_idA^i}{\sum_i \hat{\Lambda}_i}.
    \end{align*}

    Next note that for any $x \in \hat{T}$, and any pair of indices $i,j$, there exists a nearest point $x_0 \in \hat p_{ij}$ to $x$. The difference $x - x_0$ has a component $\vec{u}$ which is normal to $T\hat p_{ij}$ and a tangential component $\vec{v}$. Because $\hat p_{ij}$ is a closed submanifold, the closest point $x_0$ has the property that either $\vec{v} = 0$ or $x_0$ lies on the boundary of $\hat p_{ij}$. In the latter case the angle between $\vec{u}$ and $x - x_0$ is less than $\gamma - \frac{\pi}{2}$, where $\gamma < \pi$ is the largest interior angle of the polytope $\hat{T}$, and thus $\|\vec{v}\| \le \tan(\gamma - \frac{\pi}{2})\|\vec{u}\|$  (see Figure \ref{fig:placeholder} for a diagram). 
    \usetikzlibrary{angles,quotes}
    \begin{figure}
        \centering
        \begin{tikzpicture}
        \coordinate (A) at (-5,0);
        \coordinate (x0) at (0,0);
        \coordinate (C) at (5,-2.5);
        \coordinate (D) at (0,-2);
        \coordinate (x) at (2,-2);

        \draw (A) -- node [anchor=south] {$\hat p_{ij}$} (x0) node [anchor=south] {$x_0$};
        \draw (x0) -- (C);
        \draw [-{Stealth[scale=1]}] (x0) -- node [anchor=east] {$u$} (D);
        \draw [-{Stealth[scale=1]}] (D) -- node [anchor=north] {$v$} (x) node [anchor = north] {$x$};
        \draw (x0) -- (x);
        \draw (-0.3,0) node [anchor = north, inner sep=8pt] {$\gamma$} arc [start angle=180, end angle=333, radius = 0.3];
        \draw (0,-0.5) node [anchor = north west, inner sep=2pt] {$\theta$} arc [start angle=270, end angle = 315, radius = 0.5];
        
        \end{tikzpicture}
        \caption{A diagram of a possible configuration between $x$, a point on the interior of $\hat{T}$, and $x_0$, the nearest point in $\hat{p}_{ij}$ to $x$. The diagram pictured can be imagined as lying in the intersection of $\hat{T}$ with the plane containing $x, x_0$, and $x_0 + u$.}
        \label{fig:placeholder}
    \end{figure}
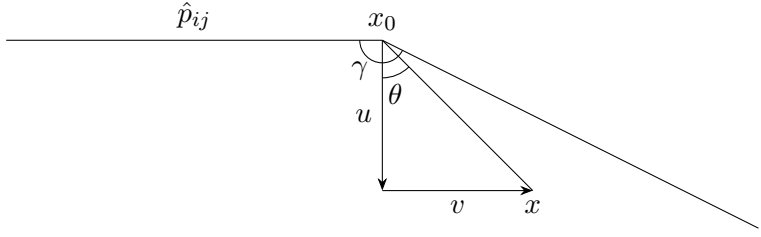
    
    In both cases, $\|x - x_0\|$ is bounded by a constant multiple of $\|\vec{u}\|$, which is in turn bounded by a constant multiple of $|\iprod{\vec{n}_i,\vec{u}}| + |\iprod{\vec{n}_j,\vec{u}}|$, since $(\vec{n}_i,\vec{n}_j)$ are linearly independent and normal to $T\hat p_{ij}$. $\lambda^i(x)$ is exactly equal to $-\iprod{\vec{n}_i,\vec{u}}$ since $x_0 \in e_i$, and similarly for $\lambda^j$. Therefore there is a constant $C$ such that $\|x - x_0\| \le C(\lambda^i(x) + \lambda^j(x))$.

    By the Taylor theorem, since $A^i(x_0) = A^j(x_0)$ and $A^i,A^j$ are both $C^r$ and bounded on $\hat{T}$, this implies there exists a constant $M$ such that $|A^i(x) - A^j(x)| \le M(\lambda^i(x) + \lambda^j(x))$.

    We can use this, together with the fact that $\|dA^i\|$ is bounded above by some constant $M'$ on $\hat{T}$ and $\|d\lambda^l\| = 1$, to bound $\|dA\|$:

    \begin{align*}\|dA\| &\le \frac{\sum_{i,j}\sum_{l \ne i}|A^i - A^j|\hat\Lambda_j\hat\Lambda_{il}\|d\lambda^l\|}{\left(\sum_j\hat\Lambda_j\right)^2} + \frac{\sum_i\hat\Lambda_i\|dA^i\|}{\sum_j\hat\Lambda_j}\\
    &\le M\frac{\sum_{i,j}\sum_{l \ne i}(\lambda^i + \lambda^j)\hat\Lambda_j\hat\Lambda_{il}}{\left(\sum_j\hat\Lambda_j\right)^2} + M'\frac{\sum_i\hat\Lambda_i}{\sum_j\hat\Lambda_j}.
    \end{align*}

    Now we can use the fact that $\hat\Lambda_j\hat\Lambda_{il}(\lambda^i + \lambda^j) = \hat\Lambda_l(\hat\Lambda_i + \hat\Lambda_j)$, so

    \[\|dA\| \le M\frac{\sum_{i,j}\sum_{l \ne i}\hat\Lambda_l(\hat\Lambda_i + \hat\Lambda_j)}{\left(\sum_j \hat\Lambda_j\right)^2} + M' \le 2MN + M'\]

    This gives an upper bound for $\|dA\|$ on the interior of $\hat T$, so $A$ is a Lipschitz function.
\end{proof}

We will use the above lemma to extend evolving frames defined on faces of a convex polytope $\hat{T}$ to the interior of the polytope. In the statement of the next lemma, the pulled back frame bundle $\Phi^*\mathcal{F}_{GL}(T)$ is the bundle over $\hat{T}$ such that the fiber over $x \in \hat{T}$ is the set of bases for $T_{\Phi(x)}(T)$. Furthermore, $S_d(\hat{T})$ refers to the codimension-$d$ stratum of $\hat{T}$, which is the union of the relative interiors of faces of $\hat{T}$ that have codimension $d$.

\begin{lemma}\label{frameextthrm}
Let $\hat{T}\subset \mathbb{R}^n$ be a closed convex polytope of dimension $n$ and $\Phi: \hat{T} \to M$ a continuous embedding that is smooth on $\hat{T} \backslash \bigcup_{d \ge 2} S_d(\hat{T})$. We will refer to the image of $\Phi$ by $T$. Let $g(t)$ be a $C^2$ (in both space and time) nondegenerate symmetric bilinear form on $T$ for $t \in [0,1]$, and let $g_0 = g(0)$. Suppose that the following are true:
\begin{enumerate} 
\item There are $C^2$ (in both space and time) frames $f(t)|_{\hat{e}_i} : \hat{e}_i \to \Phi^*\mathcal{F}_{GL}(T)$, defined on some subset $\{\hat{e}_i\}_{i = 1}^N$ of the codimension-1 faces of $\hat{T}$.
\item For each $i,j \in 1,\dots,N$, $f(t)|_{\hat{e}_i}(x) = f(t)|_{\hat{e}_j}(x)$ for all $t \in [0,1]$ and $x \in \hat{p}_{ij} = \hat{e}_i \cap \hat{e}_j$, and $\tdiff{}{t}[g(t)((f(t)|_{\hat{e}_i})_j,(f(t)|_{\hat{e}_i})_k)] = 0$ for all $i, j, k, t$.
\item There is a smooth $g_0$-orthonormal frame $f_0: \hat{T} \to \Phi^*\mathcal{F}_{GL}(T)$ such that $f_0(x) = f(0)|_{\hat{e}_i}(x)$ for all $i$ and all $x \in \hat{e}_i$. 
\end{enumerate}
Then there is a $C^2$ homotopy of frames $f(t): \hat{T} \to \Phi^*\mathcal{F}_{GL}(T)$ which is $C^2$ on $\hat{T} \backslash \bigcup_{i,j} \hat{p}_{ij}$ such that $f(t)(x) = f(t)|_{\hat{e}_i}(x)$ for all $i$ and all $x \in \hat{e}_i$, $f(0)(x) = f_0(x)$ for all $x \in \hat{T}$, $\tdiff{}{t}[g(t)(f(t)_j,f(t)_k)] = 0$ for all $j, k, t$, and $f(t)$ is Lipschitz continuous for each $t$.
\end{lemma}

\begin{proof}

    The proof will proceed as follows. We derive an ordinary differential equation for the change of basis $u: [0,1] \times \hat{T} \to GL(n)$ so that $f_0 \cdot u(t)$ is $g(t)$-orthonormal, show that it has solutions for all time, and design the free parameters of the equation so that the solution is $C^2$ and Lipschitz in space and agrees with $f(t)|_{\hat{e}_i}$.

    Suppose, first, that a frame $f(t)$ satisfies $\tdiff{}{t}[g(t)(f(t)_j,f(t)_k)] = 0$ for all $j,k$. Then let $u(t): \hat{T} \backslash \bigcup_{i,j}\hat{p}_{ij} \to GL(n)$ be the unique matrix such that $f(t) = f(0) \cdot u(t)$. The condition we have placed on $f(t)$ will allow us to find an ordinary differential equation for $u$. Let $\tilde{g}(t)_{ij} := g(t)(f(t)_i,f(t)_j)$, and $\tilde{\sigma}(t)_{ij} := \dot{g}(t)(f(0)_i,f(0)_j)$, both symmetric matrices. Then expanding out the inner products, we get
    \begin{align*}
    \pdiff{{\tilde{g}(t)}_{ij}}{t} &= \tdiff{}{t}[u^k_iu^l_jg(t)(f(0)_k,f(0)_l)] \\
    &= \dot{u}^k_iu^l_jg(t)(f(0)_k,f(0)_l) + u^k_i\dot{u}^l_jg(t)(f(0)_k,f(0)_l) \\
    &\quad\quad\quad+ u^k_iu^l_j\dot{g}(t)(f(0)_k,f(0)_l)\\
    &= (u^{-1})^m_k\dot{u}^k_ig(t)(f(t)_m,f(t)_j) + (u^{-1})^m_l\dot{u}^l_jg(t)(f(t)_i,f(t)_m) \\
    &\quad\quad\quad+ u^k_iu^l_j\tilde{\sigma}(t)_{kl}\\
    &= (u^{-1}\dot{u})^m_i\tilde{g}_{mj} + (u^{-1}\dot u )^m_j\tilde{g}_{mi} + (u^\intercal\tilde{\sigma}(t)u)_{ij} \\
    &= [\tilde{g}u^{-1}\dot{u} + (\tilde{g}u^{-1}\dot{u})^\intercal + u^\intercal\tilde{\sigma}(t)u]_{ij}\\
    &= [(\tilde{g}u^{-1}\dot{u} + \frac{1}{2}u^\intercal\tilde{\sigma}(t)u) + (\tilde{g}u^{-1}\dot{u} + \frac{1}{2}u^\intercal\tilde{\sigma}(t)u)^\intercal]_{ij}.
\end{align*}

    This shows that the condition $\pdiff{\tilde{g}}{t} = 0$ is equivalent to the condition that $\tilde{g}u^{-1}\dot{u} + \frac{1}{2}u^\intercal\tilde{\sigma}(t) u$ is a skew-symmetric matrix, which we will call $K$. Put another way, for any skew-symmetric matrix-valued map $K: [0,1] \times \hat{T} \to \mathfrak{so}(n)$, the equation 

    \begin{equation}\label{ueqn}\dot{u} = u\tilde{g}^{-1}(K - \frac{1}{2}u^\intercal\tilde{\sigma}(t) u)\end{equation}
    holds if and only if $\pdiff{\tilde{g}}{t} = 0$. Assuming $f(0)$ is $g(0)$-orthonormal, this is equivalent to saying that $f(t)$ is $g(t)$-orthonormal for all $t$, and hence $\tilde{g} = \eta =  \mathrm{diag}(1,\dots,1,-1,\dots,-1)$ with the number of negative elements equal to some number $k$ and the number of positive elements equal to $n-k$.

    In the next few paragraphs, we will argue that~\eqref{ueqn} admits a unique solution $u : [0,1] \times \hat{T} \to GL(n)$ for any Lipschitz choice of $K$. We start by constructing a solution for one particular choice of $K$, and then we leverage it to construct solutions for other choices of $K$.

    Let $\tilde{G}(t)_{ij} := g(t)({f_0}_i,{f_0}_j)$. If $X(t)$ is the LDL square root of $\tilde{G}(t)^{-1}$, so $\tilde{G}(t)^{-1} = X^T\eta X$, then $f_0 \cdot X(t)$ is a $g(t)$-orthonormal frame which is $C^2$ in space and time. Thus, $X$ satisfies equation (\ref{ueqn}) with some matrix-valued function $K': [0,1] \times \hat{T} \to \mathfrak{so}(n)$. 
    
    Now consider the differential equation
    \begin{equation}\label{Veqn}\dot{V} = V\eta K - \eta K'V.\end{equation}
    We claim that if $V:[0,1] \times \hat{T} \to O(n-k,k)$ satisfies (\ref{Veqn}), then $u = XV$ satisfies (\ref{ueqn}). This is straightforward to verify:
    \begin{align*}
        \dot{u} &= \dot{X}V + X\dot{V}\\
        &= X\eta(K' - \frac{1}{2}X^\intercal\tilde{\sigma}X)V + X(V\eta K - \eta K'V)\\
        &= X\eta K'V - \frac{1}{2}(uV^{-1})\eta(uV^{-1})^\intercal\tilde{\sigma}u  + u\eta K - X\eta K'V\\
        &= u\eta K- \frac{1}{2}uV^{-1} \eta V^{-\intercal}u^\intercal\tilde{\sigma}u \\
        &= u\eta(K - \frac{1}{2}V^\intercal V^{-\intercal}u^\intercal\tilde{\sigma}u)\\
        &= u\eta(K-\frac{1}{2}u^\intercal\tilde{\sigma}u).
    \end{align*}

    Note that for fixed $x \in \hat{T}$, the right-hand side of (\ref{Veqn}) is the time-dependent vector field $W(V,t) = V\eta K(t) - \eta K'(t)V$ on the manifold $O(n-k,k)$, and at each time it is also linear as a map $M^{n\times n}(\mathbb{R}) \to M^{n \times n}(\mathbb{R})$, so if $K, K': [0,1] \times \hat{T} \to \mathfrak{so}(n)$ are Lipschitz continuous in time then there exists a unique solution $V: [0,1] \times \hat{T} \to O(n-k,k)$ satisfying (\ref{Veqn}). If, in addition, $K$ and $K'$ are Lipschitz, $C^2$ in space, and $C^1$ in time when restricted to $[0,1] \times (\hat{T} \backslash \bigcup_{i,j} \hat{p}_{ij})$, then $V$ is $C^2$ and Lipschitz in both space and time when restricted to the same set. We already know that $K'$ satisfies all these conditions (because $X$ is in the same differentiability class as $G$), so we just need to choose $K$ appropriately.

    Since we have $C^2$ frames $f(t)|_{\hat{e}_i}$ on the faces $\hat{e}_i$, which are single-valued at $\hat{p}_{ij}$, and $\tdiff{g(t)((f(t)|_{\hat{e}_i})_j,(f(t)|_{\hat{e}_i})_k)}{t} = 0$, there are $C^2$ matrix-valued maps $u^i : [0,1] \times \hat{e}_i \to GL(n)$ and $K^i : [0,1] \times \hat{e}_i \to \mathfrak{so}(n)$ which are single-valued on $\hat p_{ij}$ and satisfy $\eta {u^i}^{-1}\dot{u}^i + \frac{1}{2}{u^i}^\intercal\tilde{\sigma}(t) u^i = K^i$. By Lemma \ref{smoothextthrm2}, the $K^i$'s can be extended (by extending each coordinate) to a continuous map $K: [0,1] \times \hat{T} \to \mathfrak{so}(n)$ which is $C^2$ and has bounded first derivatives on $[0,1] \times (\hat{T} \backslash \bigcup_{i,j}\hat{p}_{ij})$.

    This is sufficient to assert the existence of a unique family of maps $u(t): \hat{T} \to GL(n)$ satisfying $\dot{u} = u\eta (K - \frac{1}{2}u^\intercal\tilde{\sigma}(t) u)$ for $t \in [0,1]$ and $x \in \hat{T}$, $u(0) = I$, and $u(t)|_{\hat{e}_i} = u^i(t)$ (since the solution to this ordinary differential equation is unique at each $x \in \hat{e}_i$). By the smooth dependence on parameters, $u$ also has continuous second spatial derivatives when restricted to $[0,1] \times \hat{T} \backslash \bigcup_{i,j}\hat{p}_{ij}$ and is Lipschitz on $[0,1] \times \hat{T}$.
\end{proof}

\begin{corollary}\label{frameextthrm2}
Let $T$, $g(t)$, and $f_0$ be as above. Then there exists a $C^2$ homotopy of $C^2$ frames $f(t): T \to \mathcal{F}_{GL}(T)$ such that $f(0) = f_0$ and $\pdiff{}{t}g(t)(f_i(t),f_j(t)) = 0$.
\end{corollary}

\begin{proof}
In the proof of Lemma \ref{frameextthrm}, $\hat{T}$ can be set equal to $T$ and $\Phi$ can simply be the identity map. We no longer have any boundary conditions on $f(t)|_{\hat{e}_i}$. Then choosing a $C^2$ map $K: [0,1] \times T \to \mathfrak{so}(n)$ is enough to produce a solvable o.d.e.~for $u$. $K=0$ is a valid choice.
\end{proof}

\begin{proof}[Proof of Theorem \ref{compatframeexistencethrm}]
    Firstly, for each $p \subset M$, pick a frame $E_{p0}$ for $Tp$ which is $g_0$-orthonormal and let $E_p(t) = (\tau_1,\dots,\tau_{n-2})$ be a $C^2$ homotopy of $g(t)$-orthonormal frames on $p$ which are $C^2$ in space and such that $E_p(0) = E_{p0}$. These frames could be found by applying Corollary \ref{frameextthrm2}. As usual, extra $g(t)$-orthonormal vectors can be appended to $E_p(t)$ to produce frames $E_{p,e}(t) = (\tau_1,\dots,\tau_{n-2},\vec{\nu})$ and  $E_{p,e}^T(t) = (\tau_1,\dots,\tau_{n-2},\vec{\nu},\vec{n})$ which are in the same differentiability class and satisfy $E_{p,e}(0) = E_{p,e0}$ $E_{p,e}^T(0) = E_{p,e0}^T$ for each $p \subset e \subset T$. Also let $E_{e0}$ be a $g_0$-orthonormal frame on $e$ and $E_{e0}^T$ be the same frame with the outward unit normal appended, and let $A_{p,e}: p \to O(n-1)$ be the map such that $E_{e0} =E_{p,e0} \cdot A_{p,e}$, and likewise let $A_{p,e}^T: p \to O(n)$ be the map such that $E_{e0}^T =  E_{p,e0}^T \cdot A_{p,e}^T$.

\begin{figure}
    \centering
    \includegraphics[width=0.5\linewidth]{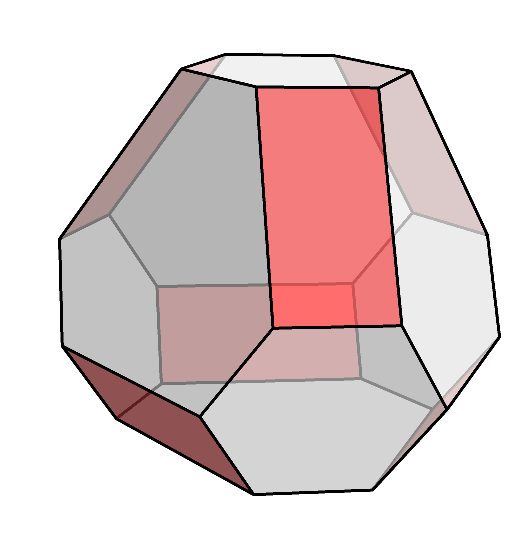}
    \caption{Blow-up of a solid tetrahedron. The regions shaded red are the sets $\overline{{\Phi^T}^{-1}(\mathring{p})}$ for the codimension-2 faces $p \subset T$. The frames $F^T(t)|_p$ are defined on these regions, and the frames $\tilde{E}_e(t)|_p$ are defined on the long sides of these regions.}
    \label{fig:blowup}
\end{figure}

    For each codimension-1 face $e$ and codimension-2 face $p \subset e$, let $\tilde{E}_e(t)|_p : \overline{{\Phi^T}^{-1}(\mathring{p})} \cap \overline{{\Phi^T}^{-1}(\mathring{e})} \to \mathcal{F}_O(e)$ be defined by the relation $\tilde{E}_e(t)|_p(x) = E_{p,e}(t)({\Phi^T}(x)) \cdot A_{p,e}(\Phi^T(x)) $, and let $\tilde{E}_{e0} : \overline{{\Phi^T}^{-1}(\mathring{e})} \to \mathcal{F}_O(e)$ be defined by $\tilde{E}_{e0}(x) = E_{e0}({\Phi^T}(x))$. Since the sets $\overline{{\Phi^T}^{-1}(\mathring{p})} \cap \overline{{\Phi^T}^{-1}(\mathring{e})}$ for differing $p$ are all disjoint and $\tilde E_e(0)|_p = \tilde E_{e0}$, by Lemma \ref{frameextthrm}, there exists a smooth map $\tilde{E}_e(t) : \overline{{\Phi^T}^{-1}(\mathring{e})} \to \mathcal{F}_O(e)$ which extends $\tilde{E}_e(t)|_p$ for each $p \subset e$ and such that $\tilde{E}_e(0) = \tilde{E}_{e0}$. The outwards normal vector can be appended to obtain $E_e^T(t)$.

    Define $F^T(t)|_e: \overline{{\Phi^T}^{-1}(\mathring{e})} \to \mathcal{F}_O(T)$ by $F^T(t)|_e(x) :=    \tilde{E}_e^T(t)(x) \cdot \mu_e^T(\Phi^T(x))$. Note that $F^T(0)|_e(x) = {\Phi^T}^*f_0(x)$ for all $x \in \overline{{\Phi^T}^{-1}(\mathring{e})}$.

    Next it is time to define $F^T(t)|_p$ for codimension-2 faces $p$ of $T$. It needs to agree with $F^T(t)|_e$ wherever both frames are defined. Due to the fact that we are keeping $\mu_e^T$ and $A_{p,e}^T$ constant and we want $F^T(t)$ to be compatible, the only question is what the function $r_{p,e}^T(t): p \to \mathbb{R}$ should be. Note that equation (\ref{discreteframejump}) is equivalent to 
    \[
    E_{p,e_j}^{T_j}(t) \cdot \exp\left(r_{p,e_j}^{T_j}(t)w^{n-1}_n\right) A_{p,e_j}^{T_j}\mu_{e_j}^{T_j} = E_{p,e_{j+1}}^{T_j}(t) \cdot A_{p,e_{j+1}}^{T_j} \mu_{e_{j+1}}^{T_j}.
    \]
    In other words, it is equivalent to $F^T(t)|_e$ being equal to $F^T(t)|_p$ wherever both are defined. Since the matrices $\mu_e^T$ and $A_{p,e}^T$ are held fixed for the entire evolution, we require $r_{p,e_j}^{T_j}(t) - \theta_p^{T_j}(t)$ to be constant. Since $r_{p,e_j}^{T_j}(0) = 0$, this means $r_{p,e_j}^{T_j}(t) = \theta^{T_j}_p(t) - \theta^{T_j}_p(0)$.

    Now we bring it all together. $F^T(t)|_p$ and $F^T(t)|_e$ are defined and $C^2$ on faces of the polytope $B_T \subseteq \mathbb{R}^n$ and single-valued everywhere, therefore they can be extended, for each $t$, to a section $F^T(t): B_T \to {\Phi^T}^*\mathcal{F}_O(T)$ which remains orthonormal for all $t$, is $C^2$ on $B_T \backslash E_{B_T}$, and has bounded derivatives on $\mathring{B}_T$. This frame is $C^2$ in time and satisfies $F^T(0) = {\Phi^T}^*f_0$. Lastly, set $f^T(t)(x) = F^T(t)(\Phi^{-1}(x))$ for $x \in S_0(T) \cup S_1(T)$. The resulting frame satisfies properties 1-3 of a $g(t)$-compatible frame by construction, and $f^T(0) = f_0$, so $f(1)$ is a compatible frame.
\end{proof}

One feature of this proof is that it suggests that the only obstruction for a frame satisfying conditions 1-3 of a compatible frame to also satisfy the fourth condition is the existence of a homotopy $g(t)$ between $g$ and a continuous metric $g_0$ such that $\theta_p^{T_j}(0) = \theta^{T_j}_p(1) - r_{p,e_j}^{T_j}(1)$ for all $p,j$.

\section{Generalized Gauss-Bonnet Theorem}\label{gengaussbonnetsection}

In this section we will work out what happens when we remove the restriction that $i_{\partial M}^*\phi = 0$, in the case $n=2$. We can derive a generalization of the Gauss-Bonnet theorem.

When specialized to 2 dimensions, the expression (\ref{distcurvfinal}) simplifies considerably, since $\mathfrak{so}(2)$ and $\Lambda^2(M)$ are both one-dimensional. In this case there exists a function $K^T: T \to \mathbb{R}$ such that $\Omega^T = K^T w^2_1 \otimes dA^T$, where $dA^T$ is the positively oriented volume form on $T$ induced by the metric $g^T$. $K^T$ is precisely the Gauss curvature, and does not depend on $f^T$. Additionally, the adjoint action is trivial on $\mathfrak{so}(2)$, and if $h$ is a smooth $SO(2)$-valued function which is a rotation by the angle $\theta$, then $h^{-1}dh$ is equal to $\begin{bmatrix}0 & -d\theta \\
d\theta & 0\end{bmatrix} = w^1_2 \otimes d\theta$.

To make use of these simplifications, let us evaluate $\iprod{\iprod{f^*\Omega_\mathrm{dist},\frac{1}{2}\phi w^2_1}}$ when $\phi$ is an arbitrary smooth function on $M$. All of the steps in Section \ref{integrationbypartssection}) are the same up until Equation (\ref{distcurvpart1}), where we need to add some boundary terms. Without loss of generality we can assume that $f^T$ is always positively oriented and, for all $e \subset \partial M$, $E_e^T = (\tau_e,\vec{n}_e)$ where $\vec{n}_e$ is the outward-pointing normal vector to $e$ and $\tau_e$ is the unit tangent vector that makes the frame positively oriented. Therefore the matrices $\mu_e^T$, when $e \subset \partial M$, are just rotations by some angle $\bar\mu_e$. 

The first boundary terms come from the integrals along codimension-1 edges $e \subset \partial M$, and they are equal to
\begin{align}
-\sum_{e \subset \partial M}\int_{\mathring{e}} \frac{1}{2}\iprod{f^*\omega \wedge \phi w^2_1} &= -\frac{1}{2}\sum_{e \subset \partial M} \int_{\mathring{e}}\iprod{ \big({\mu_e^T}^{-1}d\mu_e^T+ \mathrm{Ad}({\mu_e^T}^{-1})(\Two_e^T) \big) \wedge \phi w^2_1} \label{boundarycurv}\\
&= -\frac{1}{2}\sum_{e \subset \partial M} \int_{\mathring{e}}\phi\iprod{ \big( w^1_2 \otimes d\bar\mu_e + \Two_e^T \big) \wedge w^2_1} \notag\\
&= -\sum_{e \subset \partial M}\int_{\mathring{e}}\phi (d\bar\mu_e + (\Two_e^T)^2_1).\notag
\end{align}

The other boundary terms come from the integrals along codimension-2 points $p \in \partial M$, and following equation (\ref{angledefectpart1}), they can simply be expressed as

\begin{align*}-\sum_{p \in \partial M} \frac{1}{2}\iprod{\sum_{T \ni p}\mathrm{Ad}((A_{p,e}^T\mu_e^T)^{-1})(r_{p,e}^Tw^1_2) \wedge \phi w^2_1} &= -\frac{1}{2}\sum_{p \in \partial M} \phi(p)\iprod{\sum_{T \ni p} r_{p,e}^Tw^1_2,w^2_1}\\
&= -\sum_{p \in \partial M}\phi(p) \sum_{T \ni p}r_{p,e}^T.\end{align*}

We will now assume that the frame $f$ was constructed as in Theorem \ref{compatframeexistencethrm}, meaning the homotopy of frames is constructed such that the matrices $\mu_e^T(t)$ are constant throughout the evolution. We will also assume that a counterclockwise enumeration of the triangles $T_1,\dots,T_k$ and edges $e_1,\dots,e_{k+1}$ incident to $p \in \partial M$ is chosen much like in \ref{angledefectsection}, with the notable difference that $e_1$ and $e_{k+1}$ are only tangent to the triangles $T_1$ and $T_k$ respectively.

One of the key facts for the proof of Theorem \ref{compatframeexistencethrm} is that $r_{p,e_j}^{T_j}(t) = \theta_p^{T_j}(t) - \theta_p^{T_j}(0)$. However, since $f(0)$ is continuous, $\sum_{j = 1}^k \theta_p^{T_j}(0)$ is equal to the jump in angle between $f(0)$ and the two frames $(\tau_{e_{k+1}},\vec{n}_{e_{k+1}})$ and $(-\tau_{e_1},-\vec{n}_{e_1})$. In other words, $\sum_{j = 1}^k \theta_p^{T_j}(0) = 2\pi m_p + \pi + \bar\mu_{e_{k+1}}(0) - \bar\mu_{e_1}(0)$ for some integer $m_p$. The $2\pi m_p$ term is necessary because $\bar \mu_{e_j}(0)$ is only well defined up to addition by $2\pi$. The quantity $\bar\mu_{e_{k+1}} - \bar\mu_{e_1} + 2\pi m_p$ will from now on be shortened to $[\![\mu]\!]|_p$. Since these angles are kept fixed through the whole evolution, we get $\sum_{j = 1}^k r_{p,e_j}^{T_j} = -[\![\mu]\!]|_p - \pi + \sum_{j = 1}^k\theta_p^{T_j}$. So the additional angle defect terms take the form

\begin{equation}\label{boundaryangledefect}-\sum_{p \in \partial M} \phi(p)\sum_{T \ni p} r_{p,e}^T = \sum_{p \in \partial M}\phi(p)\left([\![\mu]\!]|_p + \pi - \sum_{T \ni p}\theta_p^T\right).\end{equation}

Synthesizing (\ref{boundarycurv}) and (\ref{boundaryangledefect}) into the expression for $\iprod{\iprod{f^*\Omega,\frac{1}{2}\phi w^2_1}}$, we can define the \emph{distributional Gauss curvature}:

\begin{align}\label{gausscurvfinal}\iprod{\iprod{K_\mathrm{dist},\phi}} &:= \iprod{\iprod{f^*\Omega_{\mathrm{dist}},\frac{1}{2}\phi w^2_1}} + \sum_{e \subset \partial M} \int_{\mathring{e}}\phi d\bar\mu_e - \sum_{p \in \partial M} \phi(p)[\![\mu]\!]|_p\\
&= \sum_{T \subseteq M} \int_{\mathring{T}}K^T\phi dA - \sum_{\mathring{e} \subset \mathring{M}}\int_{\mathring{e}} \phi [\![\Two_e]\!]^2_1 + \sum_{p \in \mathring{M}}\Theta_p \phi(p) \notag \\
&\quad\quad- \sum_{e \subset \partial M} \int_{\mathring{e}}\phi{\Two_e}^2_1 + \sum_{p \in \partial M}\left(\pi - \sum_{T \ni p}\theta_p^T\right)\phi(p).\notag\end{align}

This quantity is frame-independent and identical to the \emph{densitized distributional Gauss curvature} investigated in~\cite{gawlik2020high,gawlik2024intrinsicfiniteelementerror,BKGa22}. The fact that the frame dependence of the distributional curvature is concentrated in boundary terms is reminiscent of the following formula that is valid for a smooth metric and smooth frame, obtained from integration by parts:
\[
\int_M f^*\omega^1_2 \wedge d\phi =  \int_{\mathring{M}} \phi K dA - \int_{\partial M \backslash E_M} \phi f^*\omega^1_2.
\]

Therefore we can interpret the boundary components of the distributional curvature as being a distributional version of the connection one-form pulled back to the boundary. The analogue of $\int_{\partial M \backslash E_M} \phi f^*\omega^1_2$ above would be given by
\[\sum_{e \subset \partial M} \int_{\mathring{e}} \phi(d\bar\mu_e - kds) - \sum_{p \in \partial M}([\![\mu]\!]|_p + \pi - \sum_{T \ni p}\theta_p^T)\phi(p),
\]
where $k = \iprod{\nabla_\tau\vec{n},\tau} = -{\Two_e}^2_1(\tau)$ is the geodesic curvature of $e$ and $ds = \tau^\flat$ is the induced Riemannian length form on $e$.

Note also that, although the numbers $\bar \mu_e$ and $m_p$ are not well defined, $[\![\mu]\!]|_p$ can often be known in practical scenarios because we may impose constraints on the smooth metric that $g$ is meant to approximate. Sometimes these constraints are dictated by topology. For instance, if we know that $M$ is a manifold with boundary but no corners, then $[\![\mu]\!]|_p = 0$ for all $p$, because $\bar\mu$ must not have any discontinuities in the smooth metric. The form $d\bar\mu_e$, meanwhile, is actually well-defined for any compatible frame on any manifold. 

The distributional Gauss-Bonnet functional can be used to derive the Gauss-Bonnet theorem for compact 2-dimensional Regge manifolds:

\begin{theorem}
The following equation is true:
\[\sum_{T \subseteq M} \int_{\mathring{T}} K^T dA + \sum_{\mathring{e} \subset \mathring{M}} \int_{\mathring{e}}[\![k]\!]ds + \sum_{p \in \mathring{M}} \Theta_p + \sum_{e \subset \partial M} \int_{\mathring{e}}kds + \sum_{p \in \partial M}(\pi - \sum_{T \ni p} \theta_p^T) = 2\pi\chi(M).\]
\end{theorem}

Note that, in the case that $g$ is a smooth metric, this theorem reduces to the classical Gauss-Bonnet theorem. Also note that we are using different sign conventions for the geodesic curvature $k$ than some authors do.

\begin{proof}
    We will prove this theorem by calculating $\iprod{\iprod{K_{\mathrm{dist}},1}}$. Note that, since $[\omega,\omega] = 0$ in the case that $n = 2$, we have (for any compatible frame $f$ obtained from a smooth frame $f(0)$ as in Theorem \ref{compatframeexistencethrm})

    \begin{align*}\iprod{\iprod{K_{\mathrm{dist}},1}} &= \iprod{\iprod{f^*\Omega_\mathrm{dist},\frac{1}{2}w^2_1}} + \sum_{e \subset \partial M} \int_{\mathring{e}}d\bar\mu_e - \sum_{p \in \partial M} [\![\mu]\!]|_p \\
    &= \iprod{\iprod{df^*\omega,\frac{1}{2}w^2_1}} + \sum_{e \subset \partial M} \int_{\mathring{e}}d\bar\mu_e - \sum_{p \in \partial M} [\![\mu]\!]|_p \\
    &= \sum_{T \subseteq M} \int_{\mathring{T}} \iprod{{f^T}^*\omega \wedge \frac{1}{2}w^2_1 \otimes d(1)} + \sum_{e \subset \partial M} \int_{\mathring{e}}d\bar\mu_e 
    -\sum_{p \in \partial M} [\![\mu]\!]|_p\\
    &= \sum_{e \subset \partial M} \int_{\mathring{e}}d\bar\mu_e - \sum_{p \in \partial M} [\![\mu]\!]|_p.\end{align*}

    Meanwhile, we also have 
    \begin{align*}
    \iprod{\iprod{K_{\mathrm{dist}},1}} =& \sum_{T \subseteq M} \int_{\mathring{T}} KdA + \sum_{\mathring{e} \subseteq \mathring{M}} \int_{\mathring{e}} [\![kds]\!] + \sum_{p \in \mathring{M}} \Theta_p\\
    &+\sum_{e \subset \partial M}\int_{\mathring{e}}kds + \sum_{p \in \partial M}\left(\pi - \sum_{T \ni p} \theta_p^T\right).\\
    \end{align*}
    Therefore 
    \begin{align*}
    &\sum_{T \subseteq M} \int_{\mathring{T}} K^TdA + \sum_{\mathring{e} \subseteq \mathring{M}} \int_{\mathring{e}} [\![kds]\!] + \sum_{p \in \mathring{M}} \Theta_p + \sum_{e \subset \partial M}\int_{\mathring{e}}kds + \sum_{p \in \partial M}\left(\pi - \sum_{T \ni p} \theta_p^T\right) \\
    &=\sum_{e \subset \partial M} \int_{\mathring{e}}d\bar\mu_e - \sum_{p \in \partial M} [\![\mu]\!]|_p.
    \end{align*}

    The theorem is almost proved. All we need now is to show that $\sum_{e \subset \partial M} \int_{\mathring{e}}d\bar\mu_e - \sum_{p \in \partial M} [\![\mu]\!]|_p = 2\pi\chi(M)$. Note that $\bar\mu$ and $m_p$ are the same quantities for the smooth $g(0)$-orthonormal frame $f(0)$. The Gauss-Bonnet theorem for smooth metrics implies that, if we were to evaluate the distributional Gauss-Bonnet functional for the smooth metric $g(0)$ and smooth frame $f(0)$, we must get $2\pi\chi(M)$ for this term.
\end{proof}

\appendix

\section{Appendix: Polyhedral Manifolds, Blow-Ups, and Integration by Parts} \label{appendix}
Here we will provide a rapid overview of some key concepts that we have used, some of which are nonstandard. The primary reference for Whitney manifolds (also called ``regular manifolds'' by Whitney) and Stokes' Theorem is \cite{whitney}, with \cite{Knapp} providing a gentler introduction. The primary reference for the term ``blow-up'' as we have used it is \cite{BERCHENKOKOGAN2025100529}, although the concept was introduced earlier in at least \cite{Melrose}.

A \emph{polyhedral $n$-manifold} $M$ is a smooth manifold which is locally modeled on relatively open subsets of nondegenerate unions of parallelopipeds. This class of manifolds includes all polytopes and all domains which can be obtained by identifying faces of convex polytopes in $\mathbb{R}^n$ by rigid motions.

Specifically, every point $x \in M$ has a coordinate neighborhood $(U,\phi_U)$ where $\phi_U$ is a one-to-one open map $U \to R_U \subset \mathbb{R}^n$. $R_U = P_1\cup\dots\cup P_{N_U}$, where each set $P_i$ is a closed nondegenerate $n$-dimensional parallelopiped in $\mathbb{R}^n$ and the intersection $P_i \cap P_j$ is either empty or a shared face of both $P_i$ and $P_j$. A continuous map $f: A \subseteq R_U \to R_{U'}$, where $A$ is a relatively open subset of $R_U$, is considered a smooth map if it can be extended to a smooth map $f: \mathbb{R}^n \to \mathbb{R}^n$. The smooth structure is given by a maximal atlas of coordinate patches $(U,\phi_U)$ such that the transition maps $\phi_U \circ \phi_{U'}^{-1}$ are all smooth with this definition. As usual, $M$ is also required to be a second-countable Hausdorff space.

Note that if $\psi: A \subseteq R_U \to B \subseteq R_{U'}$ is a diffeomorphism, then each face $\Delta_d \subset \partial R_{U}$ must map to a face $\Delta_d'$ of $\partial R_{U'}$ of the same codimension. For all $0 < d \le n$, let $S_d(R_U)$ be defined as the union of relative interiors of faces of $\partial R_U$ which have codimension $d$, and define $S_0(R_U) = \mathring{R}_U$. So, if $\phi_U(x) \in S_d(R_U)$, then $\phi_{U'}(x) \in S_d(R_{U'})$ for all smooth charts $U' \ni x$. The set of such points in $M$ therefore defines a union of disjoint submanifolds (without boundary) $S_d(M) := \bigcup_{U} \phi_U^{-1}(S_d(R_U)) = \bigcup_U \bigcup_{\Delta_d \subset S_d(R_U)} \phi_U^{-1}(\mathring{\Delta}_d)$, called the $d$-stratum of $M$. The closure of a connected component of $S_d(M)$ is also a polyhedral manifold, and it is called a codimension-$d$ face of $M$, so called because it lies in a codimension-$d$ face in each coordinate chart $U$. Since $\bigcup_{d \ge 1} S_d(R_U)= \partial R_U$ for each coordinate chart $U$, we can also define $\partial M := \bigcup_{d \ge 1} S_d(M)$ and $E_M := \bigcup_{d \ge 2} S_d(M)$. The set $E_M$ is closed and has $(n-1)$-dimensional Minkowski content zero in any coordinate chart, and it will be called the exceptional set.

A version of Stokes' theorem can be produced for polyhedral manifolds, based on the fact that they satisfy all the axioms of Whitney manifolds: if $\omega$ is in $C^1\Omega^{n-1}(\mathring{M})$ and bounded on $M \backslash E_M$ (meaning each coefficient is bounded in any coordinate chart), $\omega|_{\partial M \backslash E_M}$ is summable, and $d\omega|_{\mathring{M}}$ is summable, then \cite[p.~108]{whitney}

$$\int_{\mathring{M}} d\omega = \int_{\partial M \backslash E_M} \omega = \int_{S_1(M)} \omega.$$

What we will call a \emph{Blow-Up} of a compact polyhedral manifold, called $B_M$, is any compact polyhedral manifold of the same dimension as $M$ possessing a continuous onto map of polyhedral manifolds $\Phi: B_M \to M$ such that:
\begin{enumerate}
    \item{$\Phi|_{\mathring{B}_M}: \mathring{B}_M \to \mathring{M}$ is a diffeomorphism.}
    \item{For all $d > 0$, the preimage of the relative interior of a codimension-$d$ face $\Delta_d$ of $M$ is the relative interior of a codimension-$1$ face $\hat\Delta_d$ of $B_M$, and all codimension-$1$ faces of $B_M$ are obtained this way.}
    \item{$\Phi|_{\Phi^{-1}(\mathring{e})}: \Phi^{-1}(\mathring{e}) \to \mathring{e}$ is a diffeomorphism for each codimension-$1$ face $e \subset \partial M$.}
    \item{$\Phi|_{\Phi^{-1}(\mathring{\Delta}_d)}: \Phi^{-1}(\mathring{\Delta}_d) \to \mathring{\Delta}_d$ is a smooth submersion for each codimension-$d$ face $\Delta_d \subset \partial M$, $0 < d \le n$.}
    \item{Two codimension-1 faces $\hat \Delta_d, \hat \Delta_{d'}$ of $B_M$ have non-empty intersection if and only if $\Delta_d \subseteq \Delta_{d'}$ or $\Delta_{d'} \subseteq \Delta_{d}$.}
    
\end{enumerate}
If $M$ is oriented, then we also require that $B_M$ is oriented and $\Phi$ is an orientation-preserving map.

The easiest example of a polyhedral manifold with a blow-up is the standard unit $n$-simplex embedded in $\mathbb{R}^n$. Intuitively, the blow-up is the result of intersecting a new half-hyperplane for each $d$-face, $0 \le d < n$, essentially ``cutting off'' corners, and defining the map $\Phi$ by contracting the newly created faces to the original $d$-faces. The process of cutting off corners is known as omnitruncation, and for a simplex it produces a permutahedron.

Actually writing down an explicit blow-down map between polytopes seems to be quite difficult; we are aware of \cite{matus} where a map from the blown-up simplex (the permutahedron) to the simplex is constructed that has all the needed properties, although only its inverse can be written down explicitly. We believe that a similar procedure could be used to produce blow-ups of convex polytopes which are also convex polytopes.

\section{Appendix: Relationship with Existing Formulas}\label{equivalenceappendix}

The following proposition shows that the distributional curvature defined in (\ref{distcurvendofinal}) is equivalent to the equation for the densitized distributional curvature defined by other authors, with a slightly different test space.

\begin{proposition}\label{densitizeddistcurvequiv}
    Let $f$ be a compatible frame and let $\hat{\phi}$ be a skew-symmetric, compatible $\mathrm{End}(TM)$-valued $(n-2)$-form. Then we can define a $(0,4)$-tensor $A(X,Y,Z,W) := (-1)^n\iprod{\star\phi(X,Y)W,Z}$. Then, using notation from this paper on the left side and notation from \cite{gopalakrishnan2023analysis} on the right side, the following are true at each point of $\mathring{T}$, $\mathring{e}$, and $\mathring{p}$ respectively:
    \begin{align}
    \iprod{\hat{R} \wedge \hat{\phi}} &= \frac{1}{2}\iprod{\mathcal{R},A}\omega_T, \label{equivpart1}\\
    \iprod{\hat\Two_e^T \wedge i_{\mathring{e}}^*\hat\phi|_T} &= -2\iprod{\Two^{\vec\nu^T},A_{F\hat\nu\hat\nu F}}\omega_e^T, \label{equivpart2}\\
    \iprod{\hat\Theta_p^T \wedge i_{\mathring{p}}^*\hat\phi^T} &= 2\Theta_pA_{\hat\mu\hat\nu\hat\nu\hat\mu}\omega_p.\label{equivpart3}
    \end{align}
    In addition, $A_{F\hat\nu\hat\nu F}$ is well-defined on $\mathring{e}$. 
    To be specific, $\omega_e^T$ is the induced volume form from the orientation of $T$ on $e$, and $\omega_p$ is the induced volume form from the orientation of $e$ on $T$ where $e$ is the side such that $E_{p,e}^T$ has an inwards-pointing normal vector.
\end{proposition}
\begin{proof}
    The proof strategy for all three statements is to express the left-hand side in an appropriate basis and compute. For convenience, we'll use upper indices to refer to the coframe of a corresponding frame. So for instance, $\{E_e^i\}_{i = 1}^n$ is the coframe on $e$ defined by $E_e^i({E_e}_j) = \delta^i_j$.
    
    To prove (\ref{equivpart1}) we let $\phi^i_{~jkl} := (-1)^n\iprod{\star\hat{\phi}(f_k,f_l)f_j,f_i} = A(f_k,f_l,f_i,f_j)$, so $\hat\phi = \frac{1}{2}\phi^i_{~jkl}f_i \otimes f^j \otimes \star(f^k \wedge f^l)$, and also write $\hat{R} = \frac{1}{2}\hat{R}(f_c,f_d)\otimes (f^c \wedge f^d )$. Note that these are not basis expansions for these forms, for instance the expression for $\hat{\phi}$ includes a $f_i \otimes f^j \otimes \star(f^k \wedge f^l)$ term and a $f_i \otimes f^j \otimes\star(f^l \wedge f^k)$ term for each $i,j,k,l$ (hence the $\frac{1}{2}$ factor). Then we compute:
    \begin{align*}
        \iprod{\hat{R} \wedge \hat{\phi}} &= \frac{1}{4}\iprod{\hat{R}(f_c,f_d),\phi^i_{~jkl}f_i \otimes f^j} f^c \wedge f^d \wedge \star(f^k \wedge f^l)\\
        &=\frac{1}{4}\sum_{k,l}\bigg[\iprod{\hat{R}(f_k,f_l),\phi^i_{~jkl}f_i \otimes f^j} - \iprod{\hat{R}(f_l,f_k),\phi^i_{~jkl}f_i\otimes f^j}\bigg]\omega_T\\
        &= \frac{1}{2}\sum_{j,k,l} \iprod{\hat{R}(f_k,f_l)f_i,f_j}_T\phi^i_{~jkl} \omega_T\\
        &= \frac{1}{2}\sum_{i,j,k,l} \iprod{\hat{R}(f_k,f_l)f_i,f_j}_TA(f_k,f_l,f_i,f_j) \omega_T = \frac{1}{2}\iprod{\mathcal{R},A}\omega_T.
    \end{align*}

    To prove (\ref{equivpart2}), we will use the shorthand $E_i^T = (E_e^T)_i$, where $E_e^T$ is the orthonormal frame adapted to $e$ using notation from this paper. Without the superscript $T$, $E_i$ is implicitly one of the first $n-1$ entries of $E_e^T$ which do not depend on $T$.  Let $\phi^i_{~jkl} := (-1)^n\iprod{\star_T\hat{\phi}|_T(E_k^T,E_l^T)E_j^T,E_i^T}_T = A|_T(E_k^T,E_l^T,E_i^T,E_j^T)$ so again $\hat{\phi} = \frac{1}{2}\phi^i_{~jkl} E_i^T \otimes {E^j}^T \otimes \star_T({E^k}^T \wedge {E^l}^T)$. Then we compute:

    \[\iprod{\hat\Two_e^T \wedge i_{\mathring{e}}^*\hat{\phi}|_T} = \frac{1}{2}\iprod{\hat{\Two}_e^T(E_m),\phi^i_{~jkl} E_i^T \otimes {E^j}^T} E^m \wedge i_{\mathring{e}}^*\star_T({E^k}^T \wedge {E^l}^T).\]

    Here we need to use the fact that 
    \[
    i_{\mathring{e}}^*\star_T({E^k}^T \wedge {E^l}^T) = \begin{cases}0 &\text{if } k,l \ne n,\\
    -\star_e{E^k}^T &\text{if } l = n,k \ne n\\
    \star_e{E^l}^T &\text{if } k = n,l \ne n\end{cases}
    \]
    where $\star_e$ means the Hodge star operator in the codimension-1 polytope $e$ with the orientation induced by the orientation of $T$. This can be verified using the formula for the Hodge star in an orthonormal coframe along with the fact that the volume form induced on $e$ from the orientation on $T$ is $\omega_e^T = i_{\mathring{e}}^*(E_n^T \lrcorner\omega_T)$. Then the previous line simplifies to

    \begin{align*}
        \frac{1}{2}&\iprod{\hat{\Two}_e^T(E_m),\phi^i_{~jkl} E_i^T \otimes {E^j}^T} E^m \wedge i_{\mathring{e}}^*\star_T({E^k}^T \wedge {E^l}^T) \\
        &= -\frac{1}{2}\bigg[\sum_{k}\iprod{\hat\Two_e^T(E_k),\phi^i_{~jkn}E_i^T \otimes {E^j}^T} - \sum_l\iprod{\hat\Two_e^T(E_l),\phi^i_{~jnl}E_i^T \otimes {E^j}^T}\bigg]\omega_e^T\\
        &= -\sum_k\iprod{\hat\Two_e^T(E_k),\phi^i_{~jkn}E_i^T \otimes {E^j}^T}\omega_e^T\\
        &= -\sum_{j,k}\iprod{\hat\Two_e^T(E_k)E_i^T, E_j^T}_T\phi^i_{~jkn}\omega_e^T.
    \end{align*}

    Now we use the definition that 
    \[
    \iprod{\hat\Two_e^T(E_k)E_i^T,E_j^T}_T := \begin{cases}
        0 &\text{if } i,j \ne n,\\
        \iprod{\nabla_{E_k}^TE_i^T,E_j^T}_T &\text{if } i = n~\text{or}~j = n,
    \end{cases}
    \]
    and abandon Einstein notation to avoid confusion about the $n$ indices, so the previous line is equal to
    \begin{align*}
        -\sum_{i,j,k}&\iprod{\hat\Two_e^T(E_k)E_i^T, E_j^T}_T\phi^i_{~jkn}\omega_e^T\\
        &= -\sum_{j,k} \big[\iprod{\nabla_{E_k}^TE_n^T,E_j}_T\phi^n_{~jkn} + \sum_{i,k}\iprod{\nabla_{E_k}^TE_i,E_n^T}_T\phi^i_{~nkn}\big]\omega_e^T\\
        &= -\sum_{j,k}\big[-\iprod{\nabla^T_{E_k}E_n^T,E_j}_T\phi^j_{~nkn} + \iprod{\nabla_{E_k}^TE_j,E_n^T}_T\phi^j_{~nkn}\big]\omega_e^T\\
        &= -2\sum_{j,k}\iprod{\nabla_{E_k}^TE_j,E_n^T}_T\phi^j_{~nkn}\omega_e^T\\
        &= 2\sum_{j,k} \Two^{\vec{\nu}^T}(E_k,E_j)A|_T(E_k,E_n^T,E_j,E_n^T)\omega_e^T = -2\iprod{\Two^{\vec{\nu}^T},A_{F\hat\nu\hat\nu F}}\omega_e^T.
    \end{align*}
    (Recall that in \cite{gopalakrishnan2023analysis}, $\vec\nu^T$ is the inward-pointing normal vector, whereas $E_n^T$ is the outward-pointing normal vector.)

    We still need to show that $A_{F\hat\nu\hat\nu F}$ is well-defined. Assume that $T$ is the side of $e = T \cap T'$ such that $E_e^T$ is positively oriented. Then we have
    \begin{align*}A_{F\hat\nu\hat\nu F}(E_i,E_j)|_T &= (-1)^n\iprod{\star_T\hat\phi|_T(E_i,E_n^T)E_j,E_n^T}_T \\
    &= (-1)^{i-1}\iprod{\hat\phi|_T(E_1,\dots,\hat E_i,\dots E_{n-1})E_j,E_n^T}_T\\
    &= (-1)^{i-1}{[i_{\mathring{e}}^*\hat\phi|_T(E_1,\dots,\hat E_i,\dots,E_{n-1})]_{E_e^T}}^n_j\\
    &= (-1)^{i-1}\mathrm{Ad}\left(\begin{bmatrix}I & 0 \\ 0 & -1\end{bmatrix}\right)\left([i_{\mathring{e}}^*\hat\phi|_{T'}(E_1,\dots,\hat E_i, \dots,E_{n-1})]_{E_e^{T'}}\right)^n_j\\
    &= (-1)^{i}{[i_{\mathring{e}}^*\hat\phi|_{T'}(E_1,\dots,\hat E_i,\dots,E_{n-1})]_{E_e^{T'}}}^n_j\\
    &= (-1)^i\iprod{\hat\phi|_{T'}(E_1,\dots,\hat E_i,\dots,E_{n-1})E_j,E_n^{T'}}_{T'}.
    \end{align*}
    The notation $\hat E_i$ above means the $i$th vector is skipped. Here we need to apply the fact that the frame $E_e^{T'}$ is negatively oriented on $M$, so that \break $\star_{T'}\hat\phi|_{T'}(E_i,E_n^{T'}) = (-1)^{n+i}\hat\phi|_{T'}(E_1,\dots,\hat E_i,\dots,E_{n-1})$. This shows that $A_{F \hat \nu \hat \nu F}|_{T'} = A_{F \hat \nu \hat \nu F}|_T$.

    To prove (\ref{equivpart3}), we will use the shorthand $\tau_i^T = (E_{p,e}^T)_i$, where $E_{p,e}^T$ is the positively oriented orthonormal frame adapted to $p$ using notation from this paper, and $e \supset p$ is the facet for which $\tau_n = -\vec{n}^{T}$. Without the superscript $T$, $\tau_i$ is implicitly one of the first $n-2$ entries of $E_{p,e}^T$ which do not depend on $e$ or $T$. Let $\phi^i_{~jkl} := (-1)^n\iprod{\star\hat{\phi}|_{T,e}(\tau_k^T,\tau_l^T)\tau_j^T,\tau_i^T}_T = A|_{T,e}(\tau_k^T,\tau_l^T,\tau_i^T,\tau_j^T)$, so again $\hat{\phi}|_{T,e} = \frac{1}{2}\phi^i_{~jkl} \tau_i^T \otimes {\tau^j}^T \otimes \star({\tau^k}^T \wedge {\tau^l}^T)$. 
    
    Here, because $\hat\phi|_T$ and $A|_T$ are both discontinuous on $\mathring{p}$, care needs to be taken about which of the ``representatives'' we use to evaluate $\phi^i_{~jkl}$. We used the notation $\hat\phi|_{T,e}$ to refer to the value of $\hat\phi|_T$ that is continuously extended from $\mathring{e}$, i.e. $\hat\phi|_{T,e}(x) = \lim_{m \to \infty}\hat\phi|_T(x_{m})$ where $x_{m}$ is any sequence of points in $\mathring{e}$ that converge to $x \in p$, and likewise for $A$. This choice means that $i_{\mathring{p}}^*\hat\phi|_{T,e} = {\chi_{0,p}^T}^*\hat\phi^T$, although as discussed at the end of subsection \ref{deframesection}, the product we are evaluating ultimately does not depend on this choice.
    
    Then we compute (abandoning Einstein notation):

    \begin{align*}
        \iprod{\hat\Theta_p^T \wedge {\chi_{0,p}^T}^*\hat\phi^T} &= \frac{1}{2}\Theta_p\sum_{i,j,k,l}\iprod{\tau_n^T \otimes {\tau^{n-1}}^T - \tau_{n-1}^T \otimes {\tau^n}^T, \phi^i_{~jkl}\tau_i^T \otimes {\tau^j}^T} i_{\mathring{p}}^*\star({\tau^k}^T \wedge {\tau^l}^T)
    \end{align*}
    Note that, because $\tau_n^T = -\vec{n}^T$,
    \[
    i_{\mathring{p}}^*\star({\tau^k}^T \wedge {\tau^l}^T) = \begin{cases}
        -\omega_p &\text{if } k = n-1,l = n,\\
        \omega_p &\text{if } k = n,l = n-1,\\
        0 & \text{otherwise},
    \end{cases}
    \]
    where $\omega_p = -\tau^1 \wedge\dots\wedge\tau^{n-2}$ is the volume form induced on $p$ by the orientation of $e$. So this simplifies to

    \begin{align*}
        \frac{1}{2}\Theta_p&\sum_{i,j,k,l}\iprod{\tau_n^T \otimes {\tau^{n-1}}^T - \tau_{n-1}^T \otimes {\tau^n}^T, \phi^i_{~jkl}\tau_i^T \otimes {\tau^j}^T} i_{\mathring{p}}^*\star({\tau^k}^T \wedge {\tau^l}^T)\\
        &= \frac{1}{2}\Theta_p\sum_{i,j}\iprod{\tau_n^T \otimes {\tau^{n-1}}^T - \tau_{n-1}^T \otimes {\tau^n}^T,[\phi^i_{~j,n,n-1} - \phi^i_{~j,n-1,n}]\tau_i^T \otimes {\tau^j}^T}\omega_p\\
        &= \Theta_p\sum_{i,j}\iprod{\tau_n \otimes {\tau^{n-1}}^T - \tau_{n-1} \otimes {\tau^n}^T,\phi^i_{~j,n,n-1}\tau_i^T \otimes {\tau^j}^T}\omega_p\\
        &= \Theta_p\sum_{i,j}(\delta_n^j\delta_i^{n-1} - \delta_{n-1}^j\delta_i^n)\phi^i_{~j,n,n-1}\omega_p\\
        &=\Theta_p(\phi^{n-1}_{~n,n,n-1} - \phi^n_{~n-1,n,n-1})\omega_p\\
        &=2\Theta_p\phi^{n-1}_{~n,n,n-1}\omega_p = 2\Theta_pA|_{T,e}(\tau_{n}^T,\tau_{n-1}^T,\tau_{n-1}^T,\tau_{n}^T)\omega_p = 2\Theta_pA_{\hat\mu\hat\nu\hat\nu\hat\mu}\omega_p.
    \end{align*}

\end{proof}

\begin{corollary}
    Let $\hat{\phi} \in \mathcal{A}(f,\mathcal{T},M)$ and let $A$ be defined as in Proposition \ref{densitizeddistcurvequiv}. Then, using notation from this paper on the left side and notation from \cite{gopalakrishnan2023analysis} on the right side,

    \[\iprod{\iprod{\hat{R}_{\mathrm{dist}},\hat{\phi}}} = \frac{1}{2}\widetilde{\mathcal{R}\omega}(A)\]
\end{corollary}

\begin{proof}
    The only term which is not straightforward from Proposition \ref{densitizeddistcurvequiv} and the definition of $\widetilde{\mathcal{R}\omega}$ in \cite[p.~11]{gopalakrishnan2023analysis} is the term involving the jump in second fundamental form across codimension-1 interfaces.

    In $\hat{R}_{\mathrm{dist}}$, the jump term is given by 
    
    \[-\int_{\mathring{e}}\iprod{\hat{\Two}_e^T \wedge i_{\mathring{e}}^*\hat{\phi}^T} - \iprod{\hat\Two_e^{T'} \wedge i_{\mathring{e}}^*\hat\phi^{T'}},\]
    where the orientation of the integral is chosen to be the same as the induced orientation on $e$ from $T$. Per equation (\ref{equivpart2}), this is the same as 

    \[2\int_{\mathring{e}}\iprod{\Two^{\vec\nu^T},A_{F\hat\nu\hat\nu F}}\omega_e^T - \iprod{\Two^{\vec\nu^{T'}},A_{F \hat\nu\hat\nu F}}\omega_e^{T'}.\]

    Next we note that $\omega_e^{T'} = -\omega_e^T$ because the two induced orientations are opposite, so this term simply becomes

    \[2\int_{\mathring{e}}\iprod{\Two^{\vec\nu^T} + \Two^{\vec\nu^{T'}},A_{F\hat\nu\hat\nu F}}\omega_e^T = 2\int_{\mathring{e}}\iprod{[\![\Two]\!],A_{F\hat\nu\hat\nu F}}\omega_e\]

    Note that we can drop the superscript on $\omega_e^T$ because $\omega_e$ is implicitly the volume form induced by whatever orientation the integral is being evaluated with.
\end{proof}

\section*{Acknowledgements}

We would like to thank Yasha Berchenko-Kogan for his valuable insights.  We are especially grateful for his suggestion of using the blow-up to do integration by parts.  We also acknowledge support from NSF grant DMS-2533499.

\printbibliography

\end{document}